\documentclass{amsart}
\pdfoutput=1

\usepackage{amssymb}
\usepackage{amsthm}
\usepackage{amsfonts}
\usepackage{amsmath}
\usepackage{tikz,tikz-cd}
\usetikzlibrary{math}
\usetikzlibrary{snakes}
\usepackage{pmboxdraw}
\usepackage{verbatim} 
\usepackage{graphicx}
\usepackage{color}
\usepackage[colorlinks=true, citecolor=blue, filecolor=black, linkcolor=black, urlcolor=black]{hyperref}
\usepackage{cite}
\usepackage[normalem]{ulem}
\usepackage{subcaption}
\usepackage{todonotes}
\usepackage{bbm}
\usepackage{tcolorbox}
\usepackage{marginnote}
\tcbuselibrary{breakable}
\usepackage{subcaption}

\newcommand{\RN}[1]{%
	\textup{\uppercase\expandafter{\romannumeral#1}}%
}

\def\bfs{\boldsymbol}

\def\wt{ \mathrm{wt}}
\def\Mot{ \mathrm{Mot}}
\def\erfc{ \mathrm{erfc}}

\DeclareMathOperator{\cros}{cr}
\DeclareMathOperator{\stat}{stat}
\DeclareMathOperator{\nest}{ne}
\DeclareMathOperator{\Mat}{Mat}
\DeclareMathOperator{\cl}{cl}

\def\C{\mathbb{C}}

\def\P{\mathbf{P}}
\def\R{\mathbb{R}}

\newcommand{\re}{\operatorname{Re}}

\ExplSyntaxOn

\NewDocumentCommand{\pFq}{O{}mmmmm}
 {
  \group_begin:
  \keys_set:nn { hypergeometric } { #1 }
  \hypergeometric_print:nnnnn { #2 } { #3 } { #4 } { #5 } { #6 }
  \group_end:
 }
\NewDocumentCommand{\hypergeometricsetup}{m}
 {
  \keys_set:nn { hypergeometric } { #1 }
 }

\tl_new:N \l_hypergeometric_divider_tl
\tl_new:N \l_hypergeometric_left_tl
\tl_new:N \l_hypergeometric_right_tl

\keys_define:nn { hypergeometric }
 {
  symbol .tl_set:N = \l_hypergeometric_symbol_tl,
  symbol .initial:n = F,
  separator .tl_set:N = \l_hypergeometric_separator_tl,
  separator .initial:n = {},
  skip .tl_set:N = \l_hypergeometric_skip_tl,
  skip .initial:n = 8,
  divider .choice:,
  divider/semicolon .code:n = \tl_set:Nn \l_hypergeometric_divider_tl { \;; },
  divider/bar .code:n = \tl_set:Nn \l_hypergeometric_divider_tl { \;\middle|\; },
  divider .initial:n = semicolon,
  fences .choice:,
  fences/brack .code:n = 
   \tl_set:Nn \l_hypergeometric_left_tl {[}
   \tl_set:Nn \l_hypergeometric_right_tl {]},
  fences/parens .code:n = 
   \tl_set:Nn \l_hypergeometric_left_tl {(}
   \tl_set:Nn \l_hypergeometric_right_tl {)},
  fences .initial:n = brack,
 }

\cs_new_protected:Nn \hypergeometric_print:nnnnn
 {
  {} \sb {#1} \l_hypergeometric_symbol_tl \sb { #2 }
  \left\l_hypergeometric_left_tl
  \genfrac .. 
           {0pt} 
           {} 
           { \__hypergeometric_process:n { #3 } } 
           { \__hypergeometric_process:n { #4 } } 
  \l_hypergeometric_divider_tl
  #5
  \right\l_hypergeometric_right_tl
 }

\cs_new_protected:Nn \__hypergeometric_process:n
 {
  \clist_use:nn { #1 }
   {
    {\l_hypergeometric_separator_tl}
    \mspace { \l_hypergeometric_skip_tl mu }
   }
 }

\ExplSyntaxOff

\hypergeometricsetup{
  fences=parens,
  separator={,},
  divider=bar,
}

\newcommand{\discretebinom}{\genfrac[]{0pt}{}}

\newcommand\Matching[2]{\draw (#1,0) to [out=55,in=125] (#2,0);}

\newtcolorbox{example}[1]{breakable,
colbacktitle=gray!50!white, fonttitle=\bfseries, coltitle=black, title=Example: {#1}}

\theoremstyle{plain}
\newtheorem*{thm*}{Theorem}
\newtheorem{thm}{Theorem}[section]
\newtheorem{lem}[thm]{Lemma}

\newtheorem{prop}[thm]{Proposition}
\newtheorem*{prop*}{Proposition}
\newtheorem*{lem*}{Lemma}

\theoremstyle{definition}
\newtheorem*{eg*}{Example}
\newtheorem*{egs*}{Examples}
\newtheorem{defi}[thm]{Definition}

\theoremstyle{remark}
\newtheorem*{rmk*}{Remark}
\newtheorem*{rmks*}{Remarks}

\newtheorem{rem}[thm]{Remark}

\numberwithin{equation}{section}

\begin{document}
\title[$q$-deformed Gaussian unitary ensemble]{$q$-deformed Gaussian unitary ensemble: \\ spectral moments and genus-type expansions}


\author{Sung-Soo Byun}
\address{Department of Mathematical Sciences and Research Institute of Mathematics, Seoul National University, Seoul 151-747, Republic of Korea}
\email{sungsoobyun@snu.ac.kr}

\author{Peter J. Forrester}
\address{School of Mathematical and Statistics, The University of Melbourne, Victoria 3010, Australia}
\email{pjforr@unimelb.edu.au} 

\author{Jaeseong Oh}
\address{June E Huh Center for Mathematical Challenges, Korea Institute for Advanced Study, 85 Hoegiro, Dongdaemun-gu, Seoul 02455, Republic of Korea}
\email{jsoh@kias.re.kr}


\date{\today}

\thanks{Sung-Soo Byun was partially supported by the POSCO TJ Park Foundation (POSCO Science Fellowship) and by the New Faculty Startup Fund at Seoul National University.  Funding support to Peter Forrester for this research was through the Australian Research Council Discovery Project grant DP210102887. Jaeseong Oh was partially supported by a
KIAS Individual Grant (HP083401) via the Center for Mathematical Challenges at Korea Institute for Advanced Study.
}


\begin{abstract}
The eigenvalue probability density function of the Gaussian unitary ensemble permits a $q$-extension related to the discrete $q$-Hermite weight and corresponding $q$-orthogonal polynomials. A combinatorial counting method is used to specify a positive sum formula for the spectral moments of this model. The leading two terms of the scaled $1/N^2$ genus-type expansion of the moments are evaluated explicitly in terms of the incomplete beta function. Knowledge of these functional forms allows for the smoothed leading eigenvalue density and its first correction to be determined analytically.
\end{abstract}

\maketitle

\section{Introduction}
 
\subsection{Unitary invariant ensembles and spectral moments}

For a given weight $\omega,$ we consider the eigenvalues $\bfs{x}=\{x_j\}_{j=1}^N$ of an $N \times N$ unitary invariant random matrix $X$ whose joint eigenvalue probability distribution follows
\begin{equation} \label{Gibbs}
d \P(\bfs{x}) \propto \prod_{j < k} (x_j-x_k)^2 \prod_{j=1}^N \omega(x_j) \,dx_j. 
\end{equation}
Such a random matrix is constructed according to its eigenvalue/eigenvector decomposition $X = U D U^\dagger$, where the diagonal matrix of eigenvalues is sampled according to (\ref{Gibbs}), and $U$ is a Haar distributed random unitary matrix corresponding to the eigenvectors.
The $1$-point function of $\bfs{x}$ is defined by 
\begin{equation} 
\rho_{N}(x) \equiv \rho_{N,\omega}(x) :=N \int_{ \R^{N-1} } \P(x,x_2,\dots,x_N) \,dx_2\,dx_3 \dots dx_N. 
\end{equation}
This can be expressed as 
\begin{equation} \label{spectral density gen}
\rho_{N}(x) = \sum_{j=0}^{N-1} \frac{ P_j(x)^2 }{ h_j } \omega(x) =  \frac{ P_{N}'(x) P_{N-1}(x) - P_{N-1}'(x) P_N(x)  }{h_{N-1}}  \omega(x),
\end{equation}
where $\{ P_j \}_{ j \ge 0 }$ is the sequence of \emph{monic} orthogonal polynomials satisfying the orthogonality condition
\begin{equation}
\int_\R P_j(x) P_k(x) \, \omega(x)\,dx  = h_j \, \delta_{jk}.  
\end{equation}
Here, the second expression of \eqref{spectral density gen} follows from the Christoffel-Darboux formula; see e.g. \cite[Ch.~5]{Fo10}.

The $p$-th spectral moment is defined by
\begin{equation} \label{spectral moment rhoN}
m_{N,p}:  = \mathbb{E} \Big[ \,\textup{Tr} X^{p} \Big] =  \mathbb{E}\Big[ \sum_{j=1}^N  x_j^p \Big]    = \int_\R  x^p \, \rho_N(x) \,dx, 
\end{equation}
where the expectation is taken with respect to the measure \eqref{Gibbs}.
We also denote 
\begin{equation} \label{spectral moment each term}
\mathfrak{m}_{p,j}:= \int_\R \frac{ x^p P_j(x)^2 }{ h_j } \omega(x)\,dx. 
\end{equation}
Then by \eqref{spectral density gen}, it follows that 
\begin{equation}
m_{N,p}= \sum_{ j=0 }^{N-1} \mathfrak{m}_{p,j}, \qquad  
\mathfrak{m}_{p,j} = m_{ j+1,p }- m_{ j,p }. 
\end{equation}  

\subsection{Spectral moments of the GUE}
 
As a prominent example, we discuss the spectral statistics of the Gaussian unitary ensemble (GUE).
Let us write 
\begin{equation}
\omega^{ \rm (H) }(x)= \frac{e^{-x^2/2}}{ \sqrt{2\pi} }
\end{equation}
for the standard Gaussian weight. 
With this weight, the measure \eqref{Gibbs} coincides up to a scaling factor with the eigenvalue distribution of the GUE; see e.g.~\cite[Definition 1.3.1]{Fo10} for the element-wise construction of such Hermitian random matrices.
The corresponding orthogonal polynomial is given by the Hermite polynomial
\begin{equation}
He_n(x)  = (-1)^n e^{x^2/2}\frac{d^n}{dx^n} e^{-x^2/2}, 
\end{equation}
which satisfies 
\begin{equation}
\int_\R He_n(x) He_m(x) \, \omega^{ \rm (H) }(x) \,dx= n!\,\delta_{nm}.
\end{equation}
It satisfies the three-term recurrence relation
\begin{equation} \label{Hermite three term}
He_{n+1}(x)  = x \,He_n(x)-n \, He_{n-1}(x). 
\end{equation}
By \eqref{spectral density gen}, the $1$-point function  $\rho_N^{ \rm (H) }$ of the GUE is then  given by 
\begin{align}
\begin{split}
\rho_N^{ \rm (H) } (x) & = \frac{ e^{-x^2/2} }{ \sqrt{2\pi} } \sum_{j=0}^{N-1} \frac{  He_j(x)^2   }{ j!  }
=  \frac{ e^{-x^2/2} }{\sqrt{2\pi} (N-1)! } \Big( N He_{N-1}(x)^2- (N-1) He_N(x) He_{N-2}(x) \Big). 
\label{GUE density CDI}
\end{split}
\end{align} 
We also write 
\begin{equation} \label{GUE matching}
\mathfrak{m}_{2p,j}^{ \rm (H) } = \int_\R x^{2p} \frac{He_j(x)^2}{j!} \omega^{\rm(H)}(x) \,dx, \qquad  m_{N,2p}^{ \rm (H) }= \sum_{j=0}^{N-1} \mathfrak{m}_{2p,j}^{ \rm (H) } 
\end{equation}
for the spectral moments of the GUE.

\begin{prop}[\textbf{Spectral moments of the GUE}, cf. \cite{HZ86,HT03,Ke99,FLSY23}] \label{Prop_GUE moments}
The following evaluation formulas hold true:
\begin{itemize}
    \item[(i)] \textbf{\textup{(Positive sum formula)}} We have 
    \begin{align}
     \label{GUE Matching positive}
\mathfrak{m}_{2p,j}^{ \rm (H) } &= (2p-1)!!  \sum_{l=0}^p \binom{j}{l} \binom{p}{l} 2^l, 
\\
\label{spectral moment of GUE eval}
m_{N,2p}^{ \rm (H) } &= (2p-1)!! \sum_{l=0}^{p}  \binom{N}{l+1} \binom{p}{l} 2^{l}.
    \end{align} 
     \item[(ii)] \textbf{\textup{(Alternating sum formula)}} We have  
    \begin{align}
   \label{GUE Matching alternating}
\mathfrak{m}_{2p,j}^{ \rm (H) }& =  (2p-1)!! \sum_{r=0}^{p} (-1)^{ r }  \binom{j+2p-2r}{2p}  \binom{p}{r}, 
\\
 \label{spectral moment of GUE eval signed}
 m_{N,2p}^{ \rm (H) }
 &= (2p-1)!!\,\sum_{ r=0 }^{ p-1 } (-1)^{ r } 
\bigg[ \binom{N+2p-2r-1}{2p} + \binom{N+2p-2r-2}{2p}  \bigg] 
\binom{p-1}{ r } .
    \end{align}
\end{itemize} 
\end{prop}

\begin{rem}
The formula \eqref{spectral moment of GUE eval} was shown in \cite{HZ86}, where the three-term recurrence relation 
\begin{equation} \label{HZ formula}
(p+1)m_{N,2p}^{ \rm (H) } = (4p-2)N\, m_{N,2p-2}^{ \rm (H) } +(p-1)(2p-1)(2p-3) \,m_{N,2p-4}^{ \rm (H) }
\end{equation}
was also established. 
One notes too a combinatorial proof of \eqref{spectral moment of GUE eval} due to Kerov, see \cite[Theorems 1 and 2]{Ke99}. 
In particular, in \cite[Lemma 1]{Ke99}, the quantity $\mathfrak{m}_{2p,j}^{ \rm (H) }$ was identified with the number of rook configurations on Ferrer boards; see also \cite{HSS92}.  
We also mention that the spectral moment can  be expressed as 
\begin{align}
\begin{split}
m_{N,2p}^{ \rm (H) } = (2p-1)!!\,N\, {}_2 F_1 (-p, 1-N \, ; \, 2 \, ; \,2); 
\end{split}
\end{align}
see e.g. \cite[Eqs.(4.2) and (4.3)]{CMOS19} and \cite{HT03,WF14}, where
\begin{equation}
   {}_2 F_1(a,b;c;z) := \frac{ \Gamma(c) }{\Gamma(a)\Gamma(b)} \sum_{s=0}^{\infty} \frac{\Gamma(a+s)\Gamma(b+s)}{\Gamma(c+s) \ s!}z^s 
\end{equation}
is the Gauss hypergeometric function.  
\end{rem}

\begin{rem}
Notice that two different expressions \eqref{GUE Matching positive} and \eqref{GUE Matching alternating} indicate the identity
\begin{equation}
 \sum_{l=0}^p \binom{j}{l} \binom{p}{l} 2^l= \sum_{r=0}^{p} (-1)^{ r }  \binom{j+2p-2r}{2p} \binom{p}{r}. 
\end{equation}
This identity seems to most naturally follow from a double counting; see Subsection~\ref{subsec: spectral moments of the GUE revisited}. 
\end{rem}

\begin{rem}
The first few values of the spectral moments are given by 
\begin{align}
\begin{split}
&m_{N,0}^{ \rm (H) }= N, \qquad 
m_{N,2}^{ \rm (H) }= N^2, \qquad m_{N,4}^{ \rm (H) }= 2N^3+N,
\\
&m_{N,6}^{ \rm (H) } = 5N^4+10N, \qquad 
m_{N,8}^{ \rm (H) }= 14N^5+70N^3+21N. 
\end{split}
\end{align}
On the other hand, it follows from the three-term recurrence (\ref{HZ formula}) that 
\begin{equation}\label{1.26}
\lim_{N \to \infty} \frac{1 }{ N^{ p+1 } }  m_{N,2p}^{ \rm (H) }= C_p := \frac{1}{p+1} \binom{2p}{p},
\end{equation}    
where $C_p$ is the $p$-th Catalan number. These limiting moments can be used to reconstruct the limiting scaled spectral density (see the proof of Proposition \ref{Prop_density qGUE} below) and give rise to Wigner's semi-circle law; see e.g.~\cite{AGZ09}.
\end{rem} 

\begin{rem}[Genus expansion]
The Stirling numbers of the first kind \cite[Section 26.8]{NIST} are given by
\begin{equation} \label{def of Stirling number}
s(n,k) = (-1)^{n-k} \sum_{ 1 \le b_1 <  \dots < b_{n-k} \le n-1 } b_1 b_2 \dots b_{n-k}, \qquad (n > k \ge 1).
\end{equation}
Their generating function is given by 
\begin{equation}
\sum_{k=0}^n s(n,k)x^k = \frac{ \Gamma(x+1) }{ \Gamma(x-n+1) }; 
\end{equation}
see e.g. \cite[Eq.(26.8.7)]{NIST}.
By using this and rearranging the terms in the expression (\ref{spectral moment of GUE eval}) of $m_{N,2p}^{ \rm (H) }$, we have
\begin{align*}
m_{N,2p}^{ \rm (H) } & = (2p-1)!!  \sum_{l=0}^{p} \bigg( \sum_{k=0}^{l+1} \frac{ s(l+1,k)}{(l+1)!}  N^k \bigg) \binom{p}{l} 2^{l}
\\
&=   (2p-1)!!  \sum_{r=0}^{p+1}  \bigg(  \sum_{m=0}^r  \frac{ s(p+1-m,p+1-r)}{(p+1-m)!}   \binom{p}{m} 2^{p-m} \bigg)   N^{p+1-r}.   
\end{align*}
Furthermore, it follows from the identity 
\begin{equation} \label{genus expansion; GUE}
 \sum_{m=0}^r  \frac{ s(p+1-m,p+1-r)}{(p+1-m)!}   \binom{p}{m} 2^{p-m}  =0, \qquad \textup{if } r \textup{ is odd} 
\end{equation}
that we have the $1/N^2$ expansion
\begin{equation}\label{1.30}
m_{N,2p}^{ \rm (H) } =    \sum_{g=0}^{\lfloor (p+1)/2 \rfloor}  \mathcal{E}_g(p)  N^{p+1-2g}, 
\end{equation}
where 
\begin{equation}\label{1.30a}
 \mathcal{E}_g(p)=    (2p-1)!! \sum_{m=0}^{2g}  \frac{ s(p+1-m,p+1-2g)}{(p+1-m)!}   \binom{p}{m} 2^{p-m}.   
\end{equation}
The coefficient $\mathcal{E}_g(p)$ corresponds to the count of pairings among the edges of a $2p$-gon, which is dual to a map on a compact Riemann surface of genus $g$; see \cite{HZ86}. 
Therefore, the expansion \eqref{genus expansion; GUE} is commonly referred to as the genus expansion.
In particular, the first few values of $\mathcal{E}_g$ are given as 
\begin{align}
\mathcal{E}_0(p) &= C_p , 
\qquad 
\mathcal{E}_1(p)  = C_p \, \frac{(p+1)!}{(p-2)!}  \frac{ 1}{12},  
\qquad 
\mathcal{E}_2(p)= C_p \, \frac{(p+1)!}{(p-4)!} \frac{ 5p-2 }{1440},
\\ 
\mathcal{E}_3(p) &= C_p \,\frac{(p+1)!}{(p-6)!}  \frac{ 35p^2 - 77p + 12 }{ 362880 },
\qquad 
\mathcal{E}_4(p) = C_p\,\frac{(p+1)!}{(p-8)!}  \frac{  175 p^3 - 945 p^2 + 1094 p -72  }{ 87091200 }.
\end{align}
The small order cases also follow from \cite[Theorem 7]{WF14}, which were based on a large $N$ analysis of  (\ref{HZ formula}). The form of $\mathcal{E}_g(p)$ given by (\ref{1.30a}) appears to be new. 
We also mention that  $\mathcal{E}_g(p)$ satisfies the two parameter recurrence relation \cite{HZ86}
\begin{equation}\label{1.34}
(p+2)  \mathcal{E}_{g}(p+1) = 2 (2p+1)  \mathcal{E}_g(p) +  p(2p+1)  (2p-1) \mathcal{E}_{g-1}(p-1) ,
\end{equation} 
which is also known as the topological recursion. 
\end{rem}

\subsection{Discrete $q$-deformed GUE}

We now discuss the discrete $q$-formed version of the GUE.
For this purpose, we shall frequently use the notations
\begin{align}
(a;q)_n  := (1-a)(1-aq) \dots (1-aq^{n-1}),
\qquad  (a_1,\dots, a_k ; q)_\infty := \prod_{l=0}^\infty (1-a_1 q^l) \dots (1-a_k q^l).\end{align}
The $q$-hypergeometric series ${}_r\phi_s$ is defined by
\begin{equation} \label{def of qHypergeometric}
{}_r\phi_s\left(\begin{matrix} a_1, \dots, a_r \smallskip \\ b_1, \dots, b_s \end{matrix} \,;\, q, z\right) = \sum_{k=0}^\infty \frac{ (a_1;q)_k \dots (a_r;q)_k }{ (b_1;q)_k \dots (b_s;q)_k } \frac{z^k}{ (q;q)_k } \Big( (-1)^k q^{ k(k-1)/2 } \Big)^{1+s-r}; 
\end{equation}
see e.g. \cite[Chapter 17]{NIST}.
Note that if one of the parameters $a_1, \dots, a_r$ is a non-positive power of $q$, then the sum terminates.
We also mention that as $ q \to 1$, it recovers the more familiar $ {}_rF_{s}$ hypergeometric function:  
\begin{equation}
\label{hypergeometric continumm limit}
  {}_r\phi_s\left(\begin{matrix} q^{c_1}, \dots, q^{c_r} \smallskip \\ q^{d_1}, \dots, q^{d_s} \end{matrix} \,;\, q, (q-1)^{1+s-r}\,z\right)
 \to \pFq{r}{s}{c_1,\dots,c_r}{d_1, \dots,d_s}{z} := \sum_{k=0}^\infty \frac{(c_1)_k \dots (c_r)_k }{ (d_1)_k \dots (d_s)_k } \frac{z^k}{k!} .
\end{equation}

Let us recall that the $q$-derivative $D_q$ is defined by 
\begin{equation}
D_q f(x)= \frac{f(x)-f(qx)}{(1-q)x}. 
\end{equation}
Recall also that its anti-derivative, the Jackson $q$-integral, is defined by
\begin{align}
\int_0^\infty f(x)\,d_qx  =(1-q) \sum_{k=-\infty}^\infty q^k \, f(q^k) ,
\qquad 
\int_0^\alpha f(x)\,d_qx  =(1-q)\sum_{k=0}^\infty \alpha  q^k\, f(\alpha q^k),
\end{align}
and 
\begin{equation}
\int_a^b f(x)\,d_qx  = \int_0^b f(x)\,d_q x - \int_0^a f(x)\,d_q x. 
\end{equation}
Note too that  
\begin{equation}\label{1.43}
\int_{-1}^1 f(x) \,d_q x =\frac{1}{a} \int_{-a}^a f\Big( \frac{x}{a} \Big) \,d_qx. 
\end{equation}

The discrete $q$-Hermite weight and its scaled version are defined by 
\begin{equation}\label{1.44}
\omega^{ \rm (dH) } (x) := \frac{ (qx,-qx; q)_\infty }{ (q,-1,-q;q)_\infty }, \qquad \widehat{\omega}^{ \rm (dH) } (x):=  \sqrt{1-q}\,\omega^{  {\rm (dH)} } ( \sqrt{1-q}\,x ). 
\end{equation}
The support of $\omega^{ \rm (dH) } (x)$ is taken as the $q$-lattice $\pm q^k$, $k=0,1,\dots$ (note that for these values of $x$, $\omega^{ \rm (dH) } (x) > 0$), which in the continuum limit $q \to 1^-$ is the interval $[-1,1]$. 
In this same limit, $\widehat{\omega}^{ \rm (dH) } (x)$ has support of all the real line, and moreover it recovers the Gaussian weight 
\begin{equation}
\lim_{ q \to 1^- }  \widehat{\omega}^{ \rm (dH) } (x) = \frac{1}{ \sqrt{2\pi}  }  e^{-x^2/2}; 
\end{equation}
see e.g. \cite{AAT19}. 
For $\omega^{ \rm (dH) } (x)$, one notes
that the even moments are given by 
\begin{equation}\label{1.43a}
\int_{-1}^1 x^{2p} \,\omega^{\rm (dH) } (x) \,d_qx= (1-q) (q;q^2)_p;  
\end{equation}
see e.g.~\cite{NS13},
whereas the odd moments vanish. The moments of $\widehat{\omega}^{ \rm (dH) } (x)$ follow from (\ref{1.43a}) by applying (\ref{1.43}).

We shall use several different conventions for the $q$-Hermite polynomials.
\begin{itemize}
    \item The $q$-Hermite polynomial $H_n(x;q)$ is defined by 
\begin{align}
\begin{split}
H_n(x;q) &:= x^n \, {}_2\phi_0\left(\begin{matrix} q^{-n}, q^{-n+1} \smallskip \\ - \end{matrix} \,;\, q^2, \frac{q^{2n-1}}{x^2} \right)
\\
&=   \sum_{k=0}^{n/2}(-1)^k \,\frac{ (q^{-n};q^2)_k (q^{-n+1};q^2)_k q^{k(2n-k)} }{ (q^2;q^2)_k }   x^{n-2k};
\end{split}
\end{align}
see \cite[Section 14.28]{KLS10}.
It satisfies the orthogonality
\begin{equation}
\int_{-1}^1 H_m(x;q) H_n(x;q) \omega^{ \rm (dH) } (x)  \,d_qx = (1-q) (q;q)_n q^{ n(n-1)/2 } \delta_{n,m}
\end{equation}
and the three-term recurrence relation 
\begin{equation}
x\,H_n(x;q)=H_{n+1}(x;q)  + q^{n-1} (1-q^n) H_{n-1}(x;q). 
\end{equation}
\item  We also set 
\begin{equation}
\begin{split}
\widehat{H}_n(x;q)&:= \frac{1}{(1-q)^{n/2}}  H_n\Big( \sqrt{1-q}\,x;q\Big)
\\
& =    \sum_{k=0}^{n/2} (-1)^k \,\frac{ (q^{-n};q^2)_k (q^{-n+1};q^2)_k q^{k(2n-k)} }{ (q^2;q^2)_k } \frac{1}{(1-q)^k}  x^{n-2k}.  
\end{split}
\end{equation}
Then it follows that 
\begin{equation}\label{3-term recurrence hat H}
x\,\widehat{H}_n(x;q)= \widehat{H}_{n+1}(x;q)  + q^{n-1} [n]_q\,\widehat{H}_{n-1}(x;q),
\end{equation}
where we used the convention that the $q$-analogue of an integer $n$ given by
\[
    [n]_q := \dfrac{1-q^n}{1-q}.
\]
Note that in the continuum limit, it recovers the classical Hermite polynomial  
\begin{equation}
\lim_{ q\to 1^- } \widehat{H}_n(x;q) = He_n(x).
\end{equation}
\end{itemize}

\begin{rem}
In \cite{ISV87}, the $q$-Hermite polynomial $\widetilde{H}_n(x;q)$ that satisfies the recurrence relation
\begin{equation}\label{3-term recurrence tilde H}
x\,\widetilde{H}_n(x;q)= \widetilde{H}_{n+1}(x;q) + [n]_q\,\widetilde{H}_{n-1}(x;q)
\end{equation}
has been studied.
We emphasize that while $H_n$ and $\widehat{H}_n$ are related by simple scaling, $\widetilde{H}_n$ is not. 
\end{rem}

According to (\ref{spectral density gen}) the density $\rho_{N}^{ \rm (dH) }$ of the discrete $q$-GUE is given by  
\begin{equation}
\rho_{N}^{ \rm (dH) }(x)=\frac{1}{1-q} \sum_{j=0}^{N-1} \frac{ H_j(x,q)^2 }{ (q;q)_j q^{ j(j-1)/2 }   }\omega^{ \rm (dH) } (x)  . 
\end{equation}
Similarly, we define
\begin{equation}\label{1.55}
\widehat{\rho}_N^{\rm (dH)}(x) := \sqrt{1-q}\, \rho_{N}^{ \rm (dH) }( \sqrt{1-q} \, x). 
\end{equation}
Let us define the corresponding spectral moments
\begin{align}
m_{N,2p}^{ \rm (dH) }& := \int_{-1}^1 |x|^{2p} \rho_{N}^{ \rm (dH) }(x) \,d_q x , 
\label{1.55a}\\
\widehat{m}_{N,2p}^{ \rm (dH) } &:= \int_{ -1/\sqrt{1-q} }^{1/\sqrt{1-q} } |x|^{2p} \,  \widehat{\rho}_N^{\rm (dH)}(x) \,d_q x =  \frac{1}{(1-q)^p}\,  m_{N,2p}^{ \rm (dH) }.
\end{align} 
We also denote 
\begin{equation}
\mathfrak{m}_{2p,j}^{ \rm (H) } =  \int_{-1}^1  \frac{ |x|^{2p}  H_j(x,q)^2  \omega^{ \rm (dH) } (x)  }{ (1-q)\, (q;q)_j q^{ j(j-1)/2 }   }  \,d_qx, \qquad  \widehat{\mathfrak{m}}_{2p,j}^{ \rm (H) } = \frac{1}{(1-q)^p} \mathfrak{m}_{2p,j}^{ \rm (H) }
\end{equation}
so that 
\begin{equation}
m_{N,2p}^{ \rm (dH) }= \sum_{j=0}^{N-1} \mathfrak{m}_{2p,j}^{ \rm (H) }, \qquad \widehat{m}_{N,2p}^{ \rm (dH) }= \sum_{j=0}^{N-1} \widehat{\mathfrak{m}}_{2p,j}^{ \rm (H) }.
\end{equation}

\subsection{Objectives and previous literature}
A celebrated result in the application of random matrices is that the $1/N^2$ expansion of the GUE moments (\ref{1.30}), albeit with the coefficients defined implicitly via the two-variable recurrence relation (\ref{1.34}), counts the number of pairings of the sides of a $2p$-gon which are dual to a map on a compact orientable Riemann surface of genus $g$ (equivalently, a surface of genus $g$ results when the sides of the polygon are identified in pairs). This line of work was initiated by Br\'ezin et al.~\cite{BIPZ78}, with the recursive specification of the coefficients obtained in \cite{HZ86}.
As noted below (\ref{1.26}), the leading term in the expansion (\ref{1.30}) was isolated much earlier in the development of random matrix theory, specifically as a way of computing the limiting spectral density \cite{Wi55}.

In the ensuing years eigenvalue probability distributions of the general form (\ref{Gibbs}) have occurred in various applications, notably in matrix models of low-dimensional quantum field theories \cite{Ma14}, fermionic systems \cite{DDMS19}, quantum transport \cite{Be97}, number theory \cite{KS03}, wireless communication \cite{TV04} and tiling models \cite{BG16}. Consideration of observables in these settings has motivated a number of studies into the corresponding spectral moments; references include \cite{On81,CRS06,KSS09,MS11,Wi12,MRW15,FLD16,STGK16,BK18,JKM20,ABGS21,No22,AKW22,Fo22,MS12,MS13}.

An explanation of the three-term recurrence (\ref{HZ formula}) in terms of differential properties of the Laplace transform of $\rho_N^{\rm (H)}$, first uncovered in \cite{Le04} (see also \cite{HT03}), has given motivation to the study of spectral moments in other random matrix ensembles exhibiting this and related mathematical structures \cite{Le09,WF14,CMOS19,ABGS21a,RF21,FR21,By23,BF23,CDO21,GGR21}. Considerations along these lines relate to integrable properties of the corresponding ensembles. Specifically, in \cite{MPS20} the ensemble (\ref{Gibbs}) with discrete $q$-Hermite weight (\ref{1.44}) was isolated as possessing distinguished integrable properties,
and which furthermore has the status of being
a generalisation of the GUE natural from the viewpoint of $q$-orthogonal polynomial theory. The subsequent works  \cite{FLSY23} and \cite{Co21} furthered the study of the integrable properties of this model, building on \cite{CCO20,Le05}.
We also refer to \cite{BS09,BGG17,FH22,Jo01,LSYF23,Ol21} for related studies on discrete orthogonal polynomial ensembles. 

The aim of the present work is to study integrable properties of spectral moments and the eigenvalue density of the discrete $q$-Hermite unitary invariant random matrix ensemble beyond the results of the earlier works \cite{MPS20,FLSY23,Co21}. The details of our findings are specified in the next section. They include a positive sum formula for the spectral moments \eqref{qGUE Matching positive}, the leading two terms of the scaled $1/N^2$ expansion of this formula (Theorem \ref{Thm_genus one}), and the form of the smoothed leading eigenvalue density and its leading correction (Theorem \ref{Thm_genus one density}).

\section{Main results}

\subsection{Spectral moments}

As a $q$-analogue of Proposition~\ref{Prop_GUE moments}, we have the following.

\begin{thm}[\textbf{Spectral moments of the $q$-deformed GUE}] \label{Thm_qGUE moments}
The following hold true:
\begin{itemize}
    \item[(i)] \textbf{\textup{(Positive sum formula)}} We have 
\begin{align}
\begin{split}
 \label{qGUE Matching positive}
\widehat{\mathfrak{m}}_{2p,j}^{ \rm (H) } & = \sum_{l=  0}^p  q^{(j-l)(2p-l)+l(l-1)/2} \discretebinom{j}{l}_q \dfrac{[2p]_q!}{[2p-2l]_q!! \, [l]_q!} ,
\end{split}
\\
\widehat{m}_{N,2p}^{ \rm (dH) } &= \sum_{j=0}^{N-1} \sum_{l=  0}^p  q^{(j-l)(2p-l)+l(l-1)/2} \discretebinom{j}{l}_q \dfrac{[2p]_q!}{[2p-2l]_q!! \, [l]_q!} .  \label{spectral moment of qGUE eval}
\end{align} 
Taking $q \to 1$, we reclaim \eqref{GUE Matching positive} and \eqref{spectral moment of GUE eval}. 
\smallskip 
     \item[(ii)] \textbf{\textup{(Alternating sum formula)}} We have  
\begin{align}
\begin{split}
\widehat{\mathfrak{m}}_{2p,j}^{ \rm (dH) }& = [2p-1]_q !!     \sum_{ r=0 }^{ p } (-1)^{ r } q^{ r ( r - 1  ) }  \discretebinom{j+2p-2r}{2p}_q  \discretebinom{p}{ r }_{q^2},  \label{qGUE Matching alternating}
\end{split}
\\
\widehat{m}_{N,2p}^{ \rm (dH) } 
&= [2p-1]_q !! \sum_{ r=0 }^{ p-1 } (-1)^{ r } q^{ r ( r +1  ) } 
\bigg( \discretebinom{N+2p-2r-1}{2p}_q + \discretebinom{N+2p-2r-2}{2p}_q  \bigg)
\discretebinom{p-1}{ r }_{q^2} . \label{spectral moment of qGUE eval signed}
\end{align} 
Taking $q \to 1$, we reclaim \eqref{GUE Matching alternating} and \eqref{spectral moment of GUE eval signed}. 
\end{itemize} 
\end{thm}


The positive formula \eqref{qGUE Matching positive} is new to our knowledge. 
The alternating formula \eqref{spectral moment of qGUE eval signed} was shown in \cite{Ve14} and \cite[Eq.(2.39)]{FLSY23}. 
Then \eqref{qGUE Matching alternating} follows from 
\begin{equation}
\widehat{\mathfrak{m}}_{2p,j}^{ \rm (dH) }= \widehat{m}_{j+1,2p}^{ \rm (dH) }- \widehat{m}_{j,2p}^{ \rm (dH) }. 
\end{equation}
A linear recurrence for the spectral moment $\widehat{m}_{N,2p}^{ \rm (dH) }$ extending \eqref{HZ formula} has not been discovered for general $q \in (0,1]$. However it was found by Cohen \cite[Theorem 4.4]{Co21} that a certain weighted sum over $N$ of $\widehat{\mathfrak{m}}_{2p,j}^{ \rm (dH) }$ can be evaluated in terms of a particular $q$-Hahn polynomial, from which a three-term recurrence of such summed moments follows.
 
We mention that contrary to the $q=1$ case where one can apply Pascal's triangle identity consecutively, it seems not straightforward to simplify the summation
\begin{align*}
\sum_{j=0}^{N-1} q^{j(2p-l)} \discretebinom{j}{l}_q  = \sum_{j=0}^{N-1} q^{j(2p-l)}  \frac{(1-q^j)(1-q^{j-1}) \dots (1-q^{ j-l+1 } ) }{ (1-q)(1-q^2) \dots (1-q^l) }
\end{align*}
in the expression \eqref{spectral moment of qGUE eval}.

\begin{rem}[Other forms of alternating formulas]
In \cite[Proposition 5.3]{FLSY23}, the following alternative formula has also been derived  
\begin{align}
\begin{split}
\widehat{\mathfrak{m}}_{2p,j}^{ \rm (dH) } =  \sum_{k=0}^p \sum_{l=0}^k  (-1)^k q^{k^2+2kj+ l (4l-4k-2j-1) }  (1-q)^{2l-p} \discretebinom{p}{k}_{q^2}     \discretebinom{k}{k-l}_{q^2}^2   \frac{[j]_q !}{ [j-2l]_q ! }.
\end{split}
\end{align}
It was shown in \cite[Theorem 4.3]{Co21} that the spectral moment $\widehat{m}_{N,2p}^{ \rm (dH) }$ can be expressed in terms of the hypergeometric function as 
\begin{align}
\begin{split}
\widehat{m}_{N,2p}^{ \rm (dH) } &= [2p-1]_q !! \,q^{  -\frac12 (p^2+p) } \frac{ (-q;q)_p }{ (1-q^p) (q;q)_p }   {}_3\phi_2\left(\begin{matrix} -q^{p+1},  q^p , q^{-p} \smallskip \\ -q, q^{p+1} \end{matrix} \,;\, q, q\right) 
\\
&\quad - [2p-1]_q !!  \,q^{ N p -\frac12 (p^2+p) } \frac{   (-q;q)_p }{ (1-q^p) (q;q)_p }   {}_3\phi_2\left(\begin{matrix} -q^{p+1},  q^p , q^{-p} \smallskip \\ -q, q^{p+1} \end{matrix} \,;\, q, q^{N+1}\right).  
\end{split}
\end{align}
Then by \eqref{def of qHypergeometric}, we have 
\begin{align}
\widehat{m}_{N,2p}^{ \rm (dH) } &= [2p-1]_q !! \,q^{  -\frac12 (p^2+p) } \frac{ (-q;q)_p }{ (1-q^p) (q;q)_p }  
  \sum_{k=0}^p \frac{ (-q^{p+1};q)_k \, ( q^p,q)_k \, (q^{-p};q)_k }{ (-q;q)_k \, (q^{p+1};q)_k } \frac{q^k}{ (q;q)_k }  \Big(1  - q^{(k+p)N} \Big)  ,
  \\
\widehat{\mathfrak{m}}_{2p,j}^{ \rm (dH) }   &= [2p-1]_q !! \,q^{  -\frac12 (p^2+p) } \frac{ (-q;q)_p }{ (1-q^p) (q;q)_p }  
    \sum_{k=0}^p \frac{ (-q^{p+1};q)_k \, ( q^p,q)_k \, (q^{-p};q)_k }{ (-q;q)_k \, (q^{p+1};q)_k } \frac{q^k}{ (q;q)_k } q^{(k+p)j} (1-q^{k+p}) . 
\end{align} 
\end{rem}

\subsection{Scaled moments and density on the exponential lattices}

Following the scaling for the Stieltjes-Wigert random matrix ensemble \cite{Fo21}, we set
\begin{equation} \label{q lambda scaling}
q= e^{-\lambda/N}, \qquad \lambda \ge 0.  
\end{equation} 
Recall that the regularised incomplete beta function is defined by 
\begin{equation}
I_x(a,b)=  \frac{ \Gamma(a+b) }{ \Gamma(a)\Gamma(b) } \int_0^x t^{a-1} (1-t)^{b-1}\,dt.
\end{equation}

\begin{thm}[\textbf{Genus expansion of the spectral moments}] \label{Thm_genus one}
Let $q$ be scaled as in \eqref{q lambda scaling}. 
Then for all positive integers $p$ we have 
\begin{equation} \label{genus 1 exp in thm}
q^p m_{N,2p}^{ \rm (dH) } = \Big(q(1-q)\Big)^p \widehat{m}_{N,2p}^{ \rm (dH) }=   \mathcal{M}_{2p,0}^{ \rm (dH) } N+ \frac{ \mathcal{M}_{2p,1}^{ \rm (dH) } }{ N  } +O(\frac{1}{N^3}),
\end{equation}
as $N \to \infty$, where
\begin{align}
\mathcal{M}_{2p,0}^{ \rm (dH) } & = \frac{1}{\lambda \, p}\,   I_{1-e^{-\lambda}} (p+1,p), 
\label{genus 0 coeff smoments}
\\
\mathcal{M}_{2p,1}^{ \rm (dH) } &= -\frac{\lambda\, p}{6} \bigg(  I_{1-e^{-\lambda}}(p+1,p) + \frac{(2p-1)!}{p!(p-1)!} e^{-p\lambda}(1-e^{-\lambda})^{p-1} \Big( 2+p-(2p+1)e^{-\lambda} \Big)  \bigg)  .  \label{genus 1 coeff smoments}
\end{align} 
\end{thm}

\begin{figure}[t]
  	\begin{subfigure}{0.4\textwidth}
		\begin{center}	
			\includegraphics[width=\textwidth]{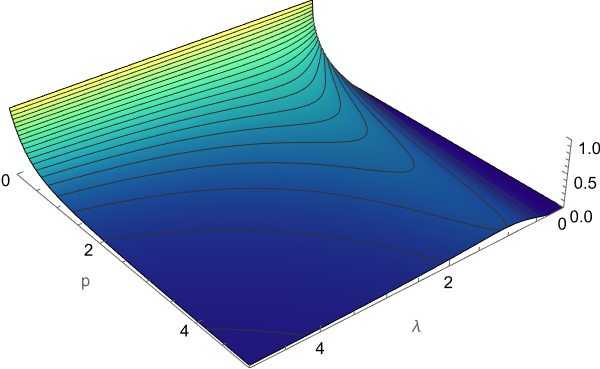}
		\end{center}
		\subcaption{$(p,\lambda) \to \mathcal{M}_{2p,0}^{ \rm (dH) }$}
	\end{subfigure}	
\qquad
	\begin{subfigure}{0.4\textwidth}
		\begin{center}	
			\includegraphics[width=\textwidth]{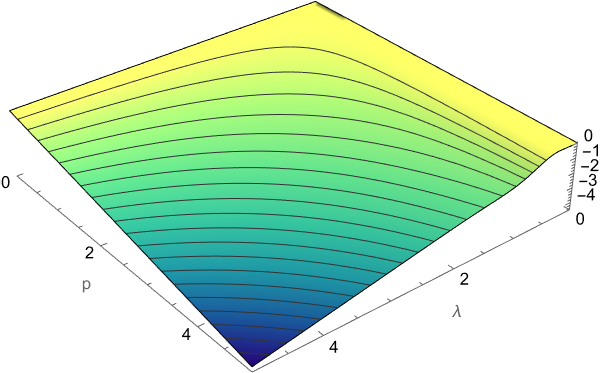}
		\end{center}
		\subcaption{$(p,\lambda) \to \mathcal{M}_{2p,1}^{ \rm (dH) }$}
	\end{subfigure}	 
 \smallskip 
 
  	\begin{subfigure}{0.4\textwidth}
		\begin{center}	
			\includegraphics[width=\textwidth]{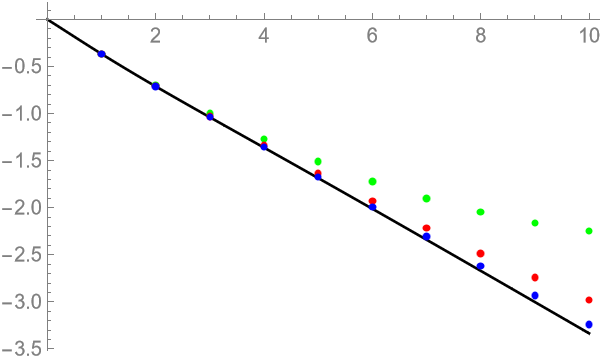}
		\end{center}
		\subcaption{$p \to \mathcal{M}_{2p,1}^{ \rm (dH) }$ with $\lambda=2$}
	\end{subfigure}	
\qquad
	\begin{subfigure}{0.4\textwidth}
		\begin{center}	
			\includegraphics[width=\textwidth]{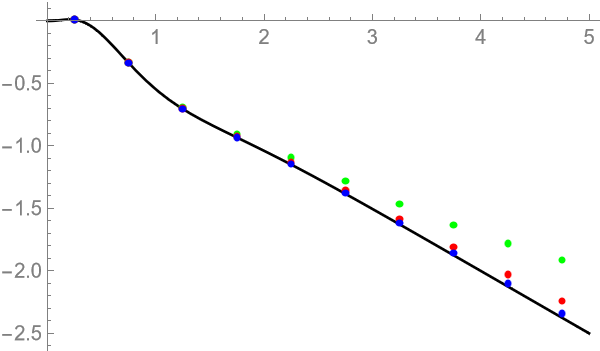}
		\end{center}
		\subcaption{$\lambda \to \mathcal{M}_{2p,1}^{ \rm (dH) }$ with $p=3$}
	\end{subfigure}	
    \caption{The plots (A) and (B) show the graphs of $\mathcal{M}_{2p,0}^{ \rm (dH) }$ and $\mathcal{M}_{2p,1}^{ \rm (dH) }$, respectively. The plot (C) shows graph of $p \to \mathcal{M}_{2p,1}^{ \rm (dH) }$ with $\lambda=2$ and its comparison with $( q^p m_{N,2p}^{ \rm (dH) } -  \mathcal{M}_{2p,0}^{ \rm (dH) } N)\, N$ for $N=10,20$, and $N=40$ (green, red and blue dots, respectively). The plot (D) shows a similar graph graph of $\lambda \to \mathcal{M}_{2p,1}^{ \rm (dH) }$ with $p=3$.}
    \label{Fig_genus1}
\end{figure}

\begin{rem}[Continuum limit $q\to 1$]
As a consequence of Theorem~\ref{Thm_genus one}, it follows that 
\begin{align}
\widehat{m}_{N,2p}^{ \rm (dH) } 
&=  \Big( \frac{N}{\lambda} \Big)^p  \bigg( \mathcal{M}_{2p,0}^{ \rm (dH) } N+   \frac{3p}{2} \lambda \, \mathcal{M}_{2p,0}^{ \rm (dH) } + \Big( \frac{(27p-1)p}{ 24 } \lambda^2\, \mathcal{M}_{2p,0}^{ \rm (dH) }   + \mathcal{M}_{2p,1}^{ \rm (dH) } \Big) \frac{1}{N} +O(\frac{1}{N^2}) \bigg). 
\end{align}
Then one can see that  
\begin{equation}
\lim_{ \lambda \to 0 } \frac{1}{\lambda^p}  \mathcal{M}_{2p,0}^{ \rm (dH) }  = \mathcal{E}_0(p)
\end{equation}
and 
\begin{equation}
\lim_{ \lambda \to 0 } \frac{1}{\lambda^p} \Big( \frac{(27p-1)p}{ 24 } \lambda^2\, \mathcal{M}_{2p,0}^{ \rm (dH) }   + \mathcal{M}_{2p,1}^{ \rm (dH) } \Big) = \frac{(2p-1)!}{ 6(p-1)! (p-2)! } = \mathcal{E}_1(p). 
\end{equation}
Thus the first two terms in the expansion (\ref{1.30}) are reclaimed.
\end{rem}

\begin{rem} 
In \cite[Eq.(2.45)]{FLSY23}, the leading scaled moment is also derived using the alternating formula, which takes the form
\begin{equation}
  \mathcal{M}_{2p,0}^{ \rm (dH) } = \frac{1}{\lambda} \bigg( \frac{1}{p} +(-1)^{p-1} \sum_{k=p}^{2p} \frac{2}{k} \frac{(2k-2)!! \,(2p-1)!! }{ (2k-2p)!! \, (k-1)!\, (2p-k)!  } e^{-\lambda k} \bigg). 
\end{equation}
Compared to this form, the formula \eqref{genus 0 coeff smoments} obtained from the positive formula \eqref{spectral moment of qGUE eval},
which can be viewed as a power series in $1 - e^{-\lambda}$, turns out to be more convenient for further analysis. 
\end{rem}

We now discuss the spectral density. 
Wigner's semi-circle law \cite{AGZ09} gives that the density $\rho_N^{ \rm (H) }$ in \eqref{GUE density CDI} satisfies  
\begin{equation}
\frac{1}{\sqrt{N}} \rho_N^{ \rm (H) }(\sqrt{N}x) \to \rho_{\rm sc}(x):=\frac{ \sqrt{4-x^2} }{2\pi} \, \mathbbm{1}_{(-2,2)}(x), \qquad N \to \infty.  
\end{equation}
On the other hand, it follows from (\ref{1.43}) and (\ref{1.55}) that the maximal support of $ \widehat{\rho}_N^{\rm (dH)}( \sqrt{N} x  )$ is 
\begin{equation}
|x| < \frac{1}{ \sqrt{N(1-q)} } = \frac{1}{\sqrt{\lambda}} +O(\frac1{N}). 
\end{equation}
Thus the limiting scaled density is supported on $|x|<1/\sqrt{\lambda}.$
Another preliminary point is that
the left-hand side of \eqref{genus 1 exp in thm} is the moment of 
\begin{equation}
\widetilde{\rho}_N^{ \rm (dH) }(x):= \sqrt{q} \, \rho_N^{ \rm (dH) }(\sqrt{q}x) = \sqrt{ \frac{q}{1-q} } \,\widehat{\rho}_N^{ \rm (dH) }\Big( \sqrt{ \frac{q}{1-q} } x \Big). 
\end{equation}

\begin{thm}[\textbf{Limiting spectral density and its finite-size correction}] 
\label{Thm_genus one density}
Let 
\begin{equation}
b(\lambda)= 2\sqrt{(1-e^{-\lambda}) e^{-\lambda} }. 
\end{equation}
Then as $N \to \infty$, in the sense of integration against a smooth function, we have
\begin{equation}
\frac{1}{N}  \widetilde{\rho}_N^{ \rm (dH) }(x) = \widetilde{\rho}_{(0)}^{ \rm (dH) }(x) + \frac{ \widetilde{\rho}_{(1)}^{ \rm (dH) }(x) }{ N^2 } + O( \frac{1}{N^4} ), \qquad |x|<1,  
\end{equation}
where 
\begin{align}\label{2.24}
\begin{split}
\widetilde{\rho}_{(0)}^{\rm (dH)}(x) & = \frac{1}{\pi \lambda x} \arg \bigg( \frac{ 2 e^{-\lambda} - x  +i \sqrt{  4(1-e^{-\lambda}) e^{-\lambda}- x^2 } }{  2 e^{-\lambda}+  x +i \sqrt{   4  (1-e^{-\lambda}) e^{-\lambda}- x^2 }  } \bigg)  \mathbbm{1}_{(-b(\lambda),b(\lambda))}(x)
\\
&\quad +
\begin{cases}
0 &\textup{if } \lambda \le \log 2,
\smallskip 
\\
\displaystyle 
\frac{1}{\lambda |x|} \Big( \mathbbm{1}_{ (-1,1) }(y) -\mathbbm{1}_{ (-b(\lambda),b(\lambda)) }(x) \Big) &\textup{if } \lambda > \log 2 
\end{cases}
\end{split}
\end{align}
and  
\begin{align}
\begin{split} \label{genus one density}
\widetilde{\rho}_{(1)}^{ \rm (dH) }(x) & = -\frac{1}{\pi} \,  \lambda \frac{ e^{-\lambda} (1-e^{-\lambda})}{2 }   (x^2-2e^{-\lambda} ) ( b(\lambda)^2-x^2 )^{-5/2}   \mathbbm{1}_{ (-b(\lambda),b(\lambda)) }(x)
\\
&\quad -\frac{1}{\pi} \frac{ \lambda }{ 3 }  \int_{ \tfrac12-x^* } ^{  \min\{ 1-e^{-\lambda}, \tfrac12+ x^*  \}  }   t \Big( x^2+2t(1-t) \Big) \Big(  4t (1-t)-x^2  \Big)^{-5/2}    \,dt .  
\end{split}
\end{align} 
Here, $x^*= \frac12 \sqrt{1-x^2}.$   
\end{thm}

As exhibited in Figure \ref{Fig_density_moment}, the graph of 
$\widetilde{\rho}_{(0)}^{\rm (dH)}(x)$ as implied by (\ref{2.24}) has a distinct shape depending on the parameter $\lambda$ being less or greater than $\log 2$.
We mention that such a phase transition in the spectral density for random matrix type models on a lattice has also been observed in earlier works; see e.g. \cite{Li12}.

The density $\widetilde{\rho}_{(0)}^{\rm (dH)}(x)$ is characterised as the probability density supported on $(-1,1)$ with even positive integer moments $ \mathcal{M}_{2p,0}^{ \rm (dH)}$. Thus
\begin{equation}\label{2.26}
\int |x|^{2p}  \widetilde{\rho}_{(0)}^{\rm (dH)}(x) \,dx =
\mathcal{M}_{2p,0}^{ \rm (dH)}.
\end{equation}
With $\mathcal{M}_{2p,0}^{ \rm (dH)}$ given by (\ref{genus 0 coeff smoments}), both sides are such that Carlson's theorem \cite[Section 5.81]{Ti39} hold, telling us that they are uniquely determined by their values at the non-negative integers. This allows the validity of (\ref{2.26}) to be extended to Re$(p) > - 1/2$.

\begin{figure}[b]
	\begin{subfigure}{0.3\textwidth}
		\begin{center}	
			\includegraphics[width=\textwidth]{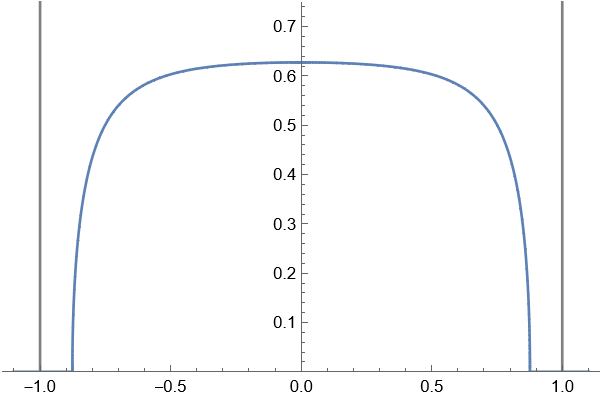}
		\end{center}
		\subcaption{$\lambda=0.3$}
	\end{subfigure}	
\quad
	\begin{subfigure}{0.3\textwidth}
		\begin{center}	
			\includegraphics[width=\textwidth]{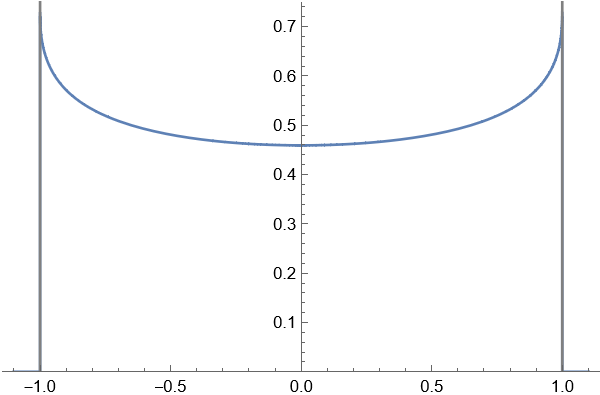}
		\end{center}
		\subcaption{$\lambda=\log 2$}
	\end{subfigure}	
 \quad
	\begin{subfigure}{0.3\textwidth}
		\begin{center}	
			\includegraphics[width=\textwidth]{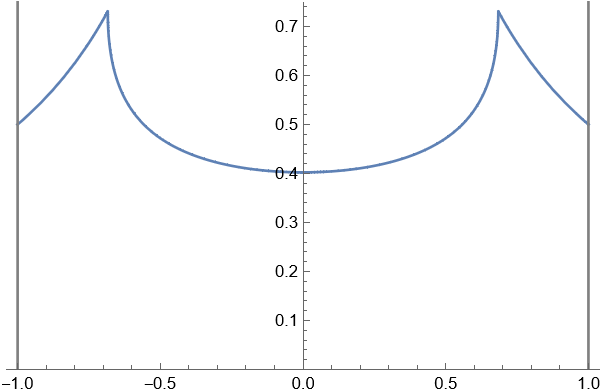}
		\end{center}
		\subcaption{$\lambda=2$}
	\end{subfigure}	

 	\begin{subfigure}{0.3\textwidth}
		\begin{center}	
			\includegraphics[width=\textwidth]{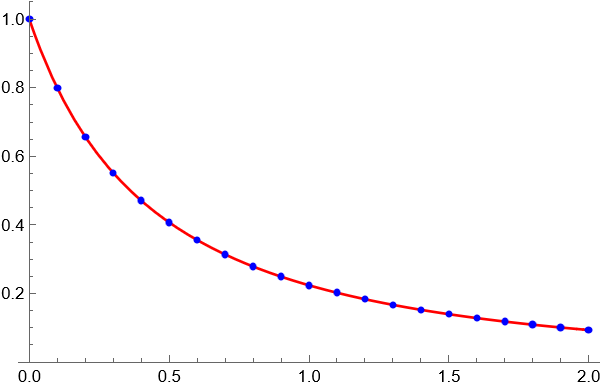}
		\end{center}
		\subcaption{$\lambda=0.3$}
	\end{subfigure}	
\quad
	\begin{subfigure}{0.3\textwidth}
		\begin{center}	
			\includegraphics[width=\textwidth]{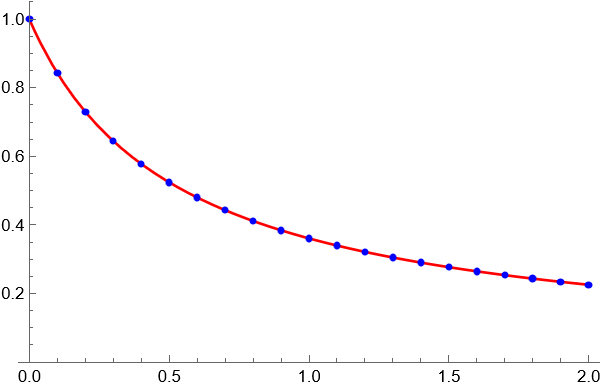}
		\end{center}
		\subcaption{$\lambda=\log 2$}
	\end{subfigure}	
 \quad
	\begin{subfigure}{0.3\textwidth}
		\begin{center}	
			\includegraphics[width=\textwidth]{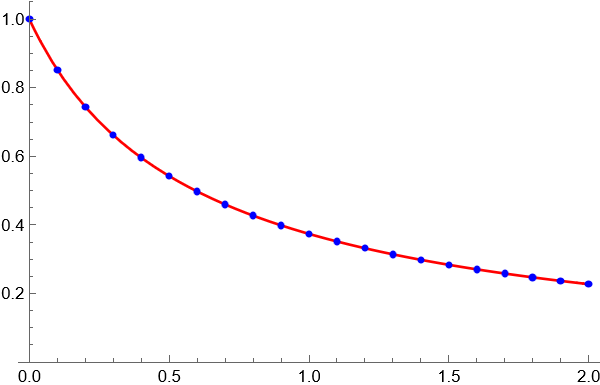}
		\end{center}
		\subcaption{$\lambda=2$}
	\end{subfigure}	
	\caption{(A)--(C) shows the spectral density $x \mapsto \widetilde{\rho}_{(0)}^{\rm (dH)}(x)$. Here the vertical line indicates the support $|x|=1$. (D)--(F) shows the graph $p \mapsto \mathcal{M}_{2p,0}^{ \rm (dH) } $ (red full line) and its comparison with $ p \mapsto  \int |x|^{2p}  \widetilde{\rho}_{(0)}^{\rm (dH)}(x) \,dx  $ (blue dots). 
 } \label{Fig_density_moment}
\end{figure}

\begin{rem}
For $\lambda<\log 2$, one can also write 
\begin{equation}
\widetilde{\rho}^{\rm (dH)}_{(0)}(x) = \frac{1}{\pi \lambda x} \bigg[ \arctan\bigg( \frac{  \sqrt{  4 e^{\lambda} -4- e^{2 \lambda} x^2 }  }{    2 - e^{\lambda} x  } \bigg) -  \arctan\bigg( \frac{  \sqrt{  4 e^{\lambda} -4- e^{2 \lambda} x^2 }  }{    2 + e^{\lambda} x   } \bigg)  \bigg] \mathbbm{1}_{(-b(\lambda),b(\lambda))}(x). 
\end{equation}
This form reveals a square-root decay at the edge, i.e. 
\begin{equation}
\widetilde{{\rho}}^{ \rm (dH) }_{(0)}(x)\asymp \sqrt{x-b(\lambda)}, \qquad x \to b(\lambda). 
\end{equation} 
\end{rem}

Due to non-integrable singularities, the integral of $\widetilde{{\rho}}_{(1)}^{ \rm (dH)} $ should be interpreted as a meromorphic continuation of a finite integral, referred to as the Hadamard regularised form; see \cite[Section III B]{WF14} for a  discussion in the context of the large $N$ expansion of the smoothed density for the Gaussian $\beta$ ensemble.

\begin{rem}[Continuum limit $q \to 1$]
One can observe that 
\begin{equation}
\sqrt{\lambda} \, \widetilde{\rho}_{(0)}( \sqrt{\lambda} x ) \to   \rho_{ \rm sc }(x), \qquad \lambda \to 0.  
\end{equation}
Therefore $ \widetilde{\rho}_{(0)} $ provides a $q$-extension of the Wigner semi-circle law. 
Note also that for $|x|<2$,
\begin{equation} \label{genus one density q1 limit}
\sqrt{\lambda} \, \widetilde{\rho}_{(1)}( \sqrt{\lambda} x ) \to    \frac{1}{\pi} (4-x^2)^{-5/2} \mathbbm{1}_{ (-2,2) }(x), \qquad \lambda \to 0. 
\end{equation}
Indeed, the first line in \eqref{genus one density} remains and the second line involving integral vanishes in the limit $\lambda \to 0$ for $|x|<2/\sqrt{\lambda}$.
The resulting $q \to 1$ limit \eqref{genus one density q1 limit} coincides with \cite[Eq.(3.18)]{WF14} with $h=0$ and $x$ restricted to $(-2,2)$. 
\end{rem}

\section{Combinatorial derivation of the spectral moments}

In this section, we review the \emph{Flajolet-Viennot theory}, a crucial combinatorial framework in this paper. Exploiting this theory, we advance towards computing spectral moments of the $q$-deformed GUE, establishing Theorem~\ref{Thm_qGUE moments}.

Let $\{ P_j \}_{ j \ge 0 }$ be the sequence of monic orthogonal polynomials satisfying  
\begin{equation}
\mathcal{L}( P_j(x) P_k(x) ) = h_j \, \delta_{jk}, 
\end{equation}
where $\mathcal{L}: \R[x] \to \R$ is the linear functional associated with the weighted Lebesgue measure $\omega(x)\,dx$, i.e. 
\begin{equation}
\mathcal{L}(P(x))= \int_\R P(x) \, \omega(x)\,dx. 
\end{equation} 
Note that $\mathfrak{m}_{p,j}$ in \eqref{spectral moment each term} can be written as 
\begin{equation}
\mathfrak{m}_{p,j} = \frac{  \mathcal{L}(  x^p \,P_j(x)^2) }{  \mathcal{L}( P_j(x)^2 )  } .
\end{equation}
In general, the family of orthogonal polynomials satisfies the three-term recurrence relation
\begin{equation}
P_{n+1}(x)=(x-b_n) P_n(x)-\lambda_n P_{n-1}(x). 
\end{equation}
Here, the orthogonal norm $h_j$ can be written in terms of the coefficients $\lambda_n$ as
\begin{equation} \label{hj lambdal}
h_j= \lambda_1 \dots \lambda_j. 
\end{equation}  

\subsection{Flajolet--Viennot theory}
Enumerating lattice paths is a classical subject in combinatorics, often serving as a combinatorial interpretation for various invariants, with moments of orthogonal polynomials being a prominent example. Flajolet \cite{Fl80} and Viennot \cite{Vi00} provided a classical combinatorial interpretation of moments using Motzkin paths. More recently, Corteel, Jonnadula, Keating and Kim \cite{CJKK23} offered another interpretation of moments using lecture hall graphs. In this section, we review the Flajolet-Viennot theory. For further exploration of this topic, we refer to \cite{CKS16} and references therein.

A (lattice) path $\omega$ is a sequence $\omega=(s_0,\dots,s_n)$, where $s_j=(x_j,y_j) \in \mathbb{N} \times \mathbb{N}.$
Each step $(s_j,s_{j+1})$ is called \emph{North-East} if $(x_{j+1},y_{j+1})=(x_j+1,y_j+1)$, \emph{East} if $(x_{j+1},y_{j+1})=(x_j+1,y_j)$, and \emph{South-East} if $(x_{j+1},y_{j+1})=(x_j+1,y_j-1)$.
We say a step $(s_j,s_{j+1})$ is of height $k$ if $y_j=k$.
\begin{defi}[\textbf{Motzkin path and weight}]
A \emph{Motzkin path} is a path consisting of North-East, East, and South-East steps which lie in the first quadrant. 
Let $\Mot_{n,k,l}$ be the set of Motzkin paths from $(0,k)$ to $(n,l).$  For given sequences $b=\{ b_n\}_{ n \ge 0 }$ and $\lambda=\{ \lambda_n \}_{ n \ge 0 }$, and a Motzkin path $\omega,$ we define the associated weight $\wt_{b,\lambda}(\omega)$ as a product of the weight of every step in $\omega$, where the weight for each step is given as 
\begin{itemize}
    \item A North-East step has weight $1$;
    \item An East step of height $k$ has weight $b_k$;
    \item A South-East step of height $k$ has weight $\lambda_k.$
\end{itemize}
\end{defi}
See Figure~\ref{Fig: Motzkin path} for an example of a Motzkin path of weight $b_0b_1b_3\lambda_1\lambda_2\lambda_3^2$. 
For a random matrix model associated with given orthogonal polynomials defined by the sequences $b_n$ and $\lambda_n$, let $\mathfrak{m}_{p,j}$ and ${m}_{N,2p}$ be the spectral moments. Then we have the following combinatorial interpretation of spectral moments in terms of Motzkin paths.

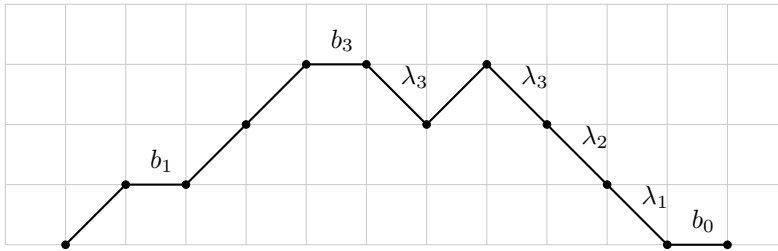
\begin{figure}[t]
    \begin{tikzpicture}[scale = 0.8]
      \tikzset{enclosed/.style={draw, circle, inner sep=0pt, minimum size=.1cm, fill=black}}
      \draw[black!20] (-1,0) grid (12,4);

      \node[enclosed] at (0,0) {};
      \node[enclosed] at (1,1) {};
      \node[enclosed] at (2,1) {};
      \node[enclosed] at (3,2) {};
      \node[enclosed] at (4,3) {};
      \node[enclosed] at (5,3) {};
      \node[enclosed] at (6,2) {};
      \node[enclosed] at (7,3) {};
      \node[enclosed] at (8,2) {};
      \node[enclosed] at (9,1) {};
      \node[enclosed] at (10,0) {};
      \node[enclosed] at (11,0) {};
      \draw[thick] (0,0) -- (1,1);
      \draw[thick] (1,1) -- (2,1);
      \draw[thick] (2,1) -- (3,2);
      \draw[thick] (3,2) -- (4,3);
      \draw[thick] (4,3) -- (5,3);
      \draw[thick] (5,3) -- (6,2);
      \draw[thick] (6,2) -- (7,3);
      \draw[thick] (7,3) -- (8,2);
      \draw[thick] (8,2) -- (9,1);
      \draw[thick] (9,1) -- (10,0);
      \draw[thick] (10,0) -- (11,0);    
      \node at (6-0.2,3-0.2) {\( \lambda_3 \)};
      \node at (8-0.2,3-0.2) {\( \lambda_3 \)};
      \node at (9-0.2,2-0.2) {\( \lambda_2 \)};
      \node at (10-0.2,1-0.2) {\( \lambda_1 \)};
      \node at (1+0.6,1+0.4) {\( b_1 \)};
      \node at (4+0.6,3+0.4) {\( b_3 \)};
      \node at (10+0.6,0.4) {\( b_0 \)};
    \end{tikzpicture}
    \caption{A Mozkin path with weight $b_0b_1b_3\lambda_1\lambda_2\lambda_3^2$}
    \label{Fig: Motzkin path}
  \end{figure}

\begin{prop}[\textbf{Combinatorial interpretation of the spectral moments}, cf. \cite{Vi00}] \label{Prop_comb rep spec moment}
Let $b,\lambda$,${\mathfrak{m}}_{p,j}$, and ${m}_{N,2p}$ as above. Then we have 
\begin{align}
\begin{split} 
{\mathfrak{m}}_{p,j} & =   \sum_{ \omega \in {\rm{Mot}}_{p,j,j}  } \wt_{b,\lambda}(\omega),
\end{split}
\\
{m}_{N,2p} &=\sum_{j=0}^{N-1} \sum_{ \omega \in {\rm{Mot}}_{p,j,j}  } \wt_{b,\lambda}(\omega) .   
\end{align} 
\end{prop}

\subsection{Spectral moments of the GUE revisited}\label{subsec: spectral moments of the GUE revisited}

Before going into the analysis of the spectral moments of the $q$-deformed GUE, it is worth reviewing the original GUE scenario (at $q=1$). The idea of using Flajolet--Viennot theory to derive the Harer-Zagier formula was originally proposed by Kerov \cite{Ke99}, employing concepts such as involutions and rook placements. While we follow Kerov's approach, we adopt the concept of `matchings' to characterize the spectral moments.

A \emph{matching} on $[n]$ is a set partition of $[n]$ such that each part has size 1 or 2. These parts are termed \emph{isolated vertices} for those containing one element and \emph{arcs} for those containing two elements. For an arc $(a,b)$ with $a<b$, the vertex $a$ is termed an \emph{opener}, and the vertex $b$ is termed a \emph{closer}. We denote the set of matchings on $[a]$ with $b$ arcs as $\Mat_{a,b}$. Typically, we represent a matching on $[n]$ as a diagram with $n$ vertices arranged horizontally, connected by arcs.

\begin{proof}[Proof of Proposition~\ref{Prop_GUE moments}]

We begin by revisiting a combinatorial interpretation of moments associated with Hermite polynomials, in terms of matching. Let $\Mat_{a,b}^{>c}\subseteq \Mat_{a,b}$ represent the set of matchings on $[a]$ with $b$ arcs, ensuring that the first $c$ vertices are either openers or isolated. Then we claim
\begin{equation}\label{eq: combinatorial interpretation of moments of Hermite}
    \mathfrak{m}_{2p,j}^{ \text{ (H) }}  = \left|\Mat_{2p+j,p}^{>j}\right|.
\end{equation}
This claim can be seen as follows. By Proposition~\ref{Prop_comb rep spec moment} along with the three-term recurrence \eqref{Hermite three term}, we obtain
\begin{equation} \label{Hermite Mot rep}
\mathfrak{m}^{\text{ (H)}}_{2p,j}  =  \sum_{ \omega \in \Mot_{2p,j,j}  } \wt_{b,\lambda}(\omega), \qquad (b_n=0,\quad \lambda_n=n).  
\end{equation}
Given $b_n=0$, our focus narrows down to the Motzkin paths without an East-step, so-called the \emph{Dyck paths}. We define a \emph{Hermite history} as a labeled Dyck path from $(0,j)$ to $(2p,j)$ where the labeling of its South-East step of height $k$ is labeled with an integer in $[k]$. 
An illustration of a Hermite history is presented on the left side of the bottom of Figure~\ref{Fig: Motzkin path and matching}. It is evident that for a Dyck path $\omega$, there are $\wt_{b,\lambda}(\omega)$-many \emph{Hermite histories} obtained by labeling South-East steps of $\omega$.

Subsequently, we establish a bijective correspondence between the set of Hermite histories from $(0,j)$ to $(2p,j)$ and $\Mat_{2p+j,p}^{>j}$, in the following dynamic manner.
\begin{itemize}
\item Start with $j$ vertices (depicted as blue dots in the right side of Figure~\ref{Fig: Motzkin path and matching}).
\smallskip
\item Progressing through each step of the Dyck path from $(0,j)$, a vertex is added (depicted as red dots in Figure~\ref{Fig: Motzkin path and matching}). Consequently, a total of $2p$ vertices are appended.
\smallskip
\item When the Hermite history advances with a South-East step labeled by $m$, we introduce a new arc connecting the newly added vertex and the $m$-th isolated point, counted from the left. If it advances with a North-East step, a positive vertex is added without any accompanying arc.
\end{itemize}
See Figure~\ref{Fig: Motzkin path and matching} for a visual representation of the above bijection.

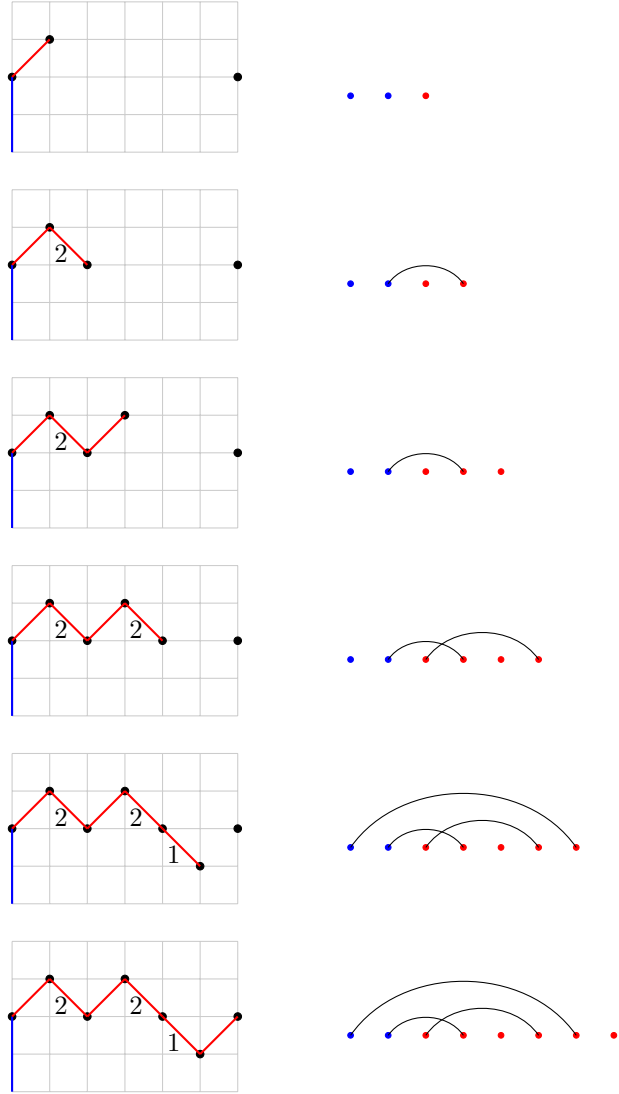
\begin{figure}[h!]    

\begin{center}
      \begin{tikzpicture}[scale = 0.5]
      \tikzset{enclosed/.style={draw, circle, inner sep=0pt, minimum size=.1cm, fill=black}}

   \draw[black!20] (0,25) grid (6,29);
      \node[enclosed] at (0,27) {};
        \node[enclosed] at (1,28) {};
     \node[enclosed] at (6,27) {}; 

      \draw[thick,blue] (0,25) -- (0,27);
      \draw[thick,red] (0,27) -- (1,28);
      

    \draw[black!20] (0,20) grid (6,24);
        \node[enclosed] at (0,22) {};
        \node[enclosed] at (1,23) {};
                \node[enclosed] at (2,22) {};
     \node[enclosed] at (6,22) {}; 

       \draw[thick,blue] (0,20) -- (0,22);
      \draw[thick,red] (0,22) -- (1,23);
           \draw[thick,red] (1,23) -- (2,22);
      
      \node at (1+0.3,22+0.3) {\( 2 \)};

    \draw[black!20] (0,15) grid (6,19);
      \node[enclosed] at (0,17) {};
        \node[enclosed] at (1,18) {};
                \node[enclosed] at (2,17) {};
                \node[enclosed] at (3,18) {};
     \node[enclosed] at (6,17) {}; 

      \draw[thick,blue] (0,15) -- (0,17);
      \draw[thick,red] (0,17) -- (1,18);
           \draw[thick,red] (1,18) -- (2,17);
                \draw[thick,red] (2,17) -- (3,18);
      
      \node at (1+0.3,17+0.3) {\( 2 \)};
  
     \draw[black!20] (0,10) grid (6,14);
      \node[enclosed] at (0,12) {};
        \node[enclosed] at (1,13) {};
                \node[enclosed] at (2,12) {};
                \node[enclosed] at (3,13) {};
                       \node[enclosed] at (4,12) {};
     \node[enclosed] at (6,12) {}; 

        \draw[thick,blue] (0,10) -- (0,12);
      \draw[thick,red] (0,12) -- (1,13);
           \draw[thick,red] (1,13) -- (2,12);
                \draw[thick,red] (2,12) -- (3,13);
                     \draw[thick,red] (3,13) -- (4,12);
      
      \node at (1+0.3,12+0.3) {\( 2 \)};
        \node at (3+0.3,12+0.3) {\( 2 \)};

      \draw[black!20] (0,5) grid (6,9);
      \node[enclosed] at (0,7) {};
        \node[enclosed] at (1,8) {};
                \node[enclosed] at (2,7) {};
                \node[enclosed] at (3,8) {};
                       \node[enclosed] at (4,7) {};
                       \node[enclosed] at (5,6) {};
     \node[enclosed] at (6,7) {}; 
      
         \draw[thick,blue] (0,5) -- (0,7);
      \draw[thick,red] (0,7) -- (1,8);
           \draw[thick,red] (1,8) -- (2,7);
                \draw[thick,red] (2,7) -- (3,8);
                     \draw[thick,red] (3,8) -- (4,7);
                          \draw[thick,red] (4,7) -- (5,6);
      \node at (1+0.3,7+0.3) {\( 2 \)};
        \node at (3+0.3,7+0.3) {\( 2 \)};
        \node at (4+0.3,6+0.3) {\( 1 \)};
        
      \draw[black!20] (0,0) grid (6,4);
      \node[enclosed] at (0,2) {};
        \node[enclosed] at (1,3) {};
                \node[enclosed] at (2,2) {};
                \node[enclosed] at (3,3) {};
                       \node[enclosed] at (4,2) {};
                       \node[enclosed] at (5,1) {};
     \node[enclosed] at (6,2) {}; 

        \draw[thick,blue] (0,0) -- (0,2);
      \draw[thick,red] (0,2) -- (1,3);
           \draw[thick,red] (1,3) -- (2,2);
                \draw[thick,red] (2,2) -- (3,3);
                     \draw[thick,red] (3,3) -- (4,2);
                          \draw[thick,red] (4,2) -- (5,1);
                               \draw[thick,red] (5,1) -- (6,2);
      \node at (1+0.3,2+0.3) {\( 2 \)};
        \node at (3+0.3,2+0.3) {\( 2 \)};
        \node at (4+0.3,1+0.3) {\( 1 \)};

       \draw [fill,blue] (9,26.5) circle [radius=0.075] ;
             \draw [fill,blue] (10,26.5) circle [radius=0.075] ;
              \draw [fill,red] (11,26.5) circle [radius=0.075] ;
               

      \draw [fill,blue] (9,21.5) circle [radius=0.075] ;
             \draw [fill,blue] (10,21.5) circle [radius=0.075] ;
              \draw [fill,red] (11,21.5) circle [radius=0.075] ;
               \draw [fill,red] (12,21.5) circle [radius=0.075] ;
               
                      \draw (10,21.5) to [out=55,in=125] (12,21.5);
                     

      \draw [fill,blue] (9,16.5) circle [radius=0.075] ;
             \draw [fill,blue] (10,16.5) circle [radius=0.075] ;
              \draw [fill,red] (11,16.5) circle [radius=0.075] ;
               \draw [fill,red] (12,16.5) circle [radius=0.075] ;
                \draw [fill,red] (13,16.5) circle [radius=0.075] ;
               
                      \draw (10,16.5) to [out=55,in=125] (12,16.5);

      \draw [fill,blue] (9,11.5) circle [radius=0.075] ;
             \draw [fill,blue] (10,11.5) circle [radius=0.075] ;
              \draw [fill,red] (11,11.5) circle [radius=0.075] ;
               \draw [fill,red] (12,11.5) circle [radius=0.075] ;
                \draw [fill,red] (13,11.5) circle [radius=0.075] ;
                 \draw [fill,red] (14,11.5) circle [radius=0.075] ;
               
                      \draw (10,11.5) to [out=55,in=125] (12,11.5);
                      \draw (11,11.5) to [out=55,in=125] (14,11.5);

       \draw [fill,blue] (9,6.5) circle [radius=0.075] ;
             \draw [fill,blue] (10,6.5) circle [radius=0.075] ;
              \draw [fill,red] (11,6.5) circle [radius=0.075] ;
               \draw [fill,red] (12,6.5) circle [radius=0.075] ;
                \draw [fill,red] (13,6.5) circle [radius=0.075] ;
                 \draw [fill,red] (14,6.5) circle [radius=0.075] ;
                  \draw [fill,red] (15,6.5) circle [radius=0.075] ;
               
                    \draw (9,6.5) to [out=55,in=125] (15,6.5);
                      \draw (10,6.5) to [out=55,in=125] (12,6.5);
                      \draw (11,6.5) to [out=55,in=125] (14,6.5);
      
          \draw [fill,blue] (9,1.5) circle [radius=0.075] ;
             \draw [fill,blue] (10,1.5) circle [radius=0.075] ;
              \draw [fill,red] (11,1.5) circle [radius=0.075] ;
               \draw [fill,red] (12,1.5) circle [radius=0.075] ;
                \draw [fill,red] (13,1.5) circle [radius=0.075] ;
                 \draw [fill,red] (14,1.5) circle [radius=0.075] ;
                  \draw [fill,red] (15,1.5) circle [radius=0.075] ;
                   \draw [fill,red] (16,1.5) circle [radius=0.075] ;
               
                    \draw (9,1.5) to [out=55,in=125] (15,1.5);
                      \draw (10,1.5) to [out=55,in=125] (12,1.5);
                      \draw (11,1.5) to [out=55,in=125] (14,1.5);

    \end{tikzpicture}
\end{center}
    \caption{A Hermite history from $(0,2)$ to $(6,2)$ its corresponding matching}
    \label{Fig: Motzkin path and matching}
  \end{figure}

\medskip 

For the remainder of the proof, we enumerate $|\Mat_{2p+j,p}^{>j}|$ in two distinct ways, leading to \eqref{GUE Matching positive} and \eqref{GUE Matching alternating}, respectively.

\medskip 

We first show \eqref{GUE Matching positive}. Let $\Mat_{2p+j,p}^{>j}(i)$ be the set of matchings in $\Mat_{2p+j,p}^{>j}$ such that there are $i$ arcs connecting to the first $j$ vertices. Or equivalently, $j-i$ vertices among the first $j$ vertices are isolated. We claim that 
\begin{equation}\label{eq: refined Hermite count}
    \left|\Mat_{2p+j,p}^{>j}(i)\right| = (2p-1)!! \binom{j}{i}  \binom{p}{i} 2^{i},
\end{equation}
thereby completing the proof. To that end, we construct matchings in $\Mat_{2p+j,p}^{>j}(i)$ step by step as follows.
\begin{itemize}
    \item Firstly, let us consider a matching in $\Mat_{2p-i,p-i}$. The count of such matchings is given by
    \[
    \binom{2p-i}{i} (2p-2i-1)!!.
    \]
    Here, we select $2p-2i$ vertices out of $2p-i$ available vertices and form perfect matchings with $p-i$ arcs on the chosen vertices.
    \smallskip
    \item For each matching $M\in\Mat_{2p-i,p-i}$, we introduce $2i$ new vertices, with the first $i$ vertices placed in the leftmost part. Subsequently, we create $i$ arcs among newly introduced $2i$ vertices, connecting each of the first $i$ vertices with the remaining $i$ vertices. The total count of such configurations is given by
    \[
    \binom{2p}{i} i! = \dfrac{(2p)!}{(2p-i)!}.
    \]
    \item Finally, we introduce $j-i$ isolated vertices in between the first $i$ vertices so that in the resulting matching, newly added vertices are in the first $j$ vertices. The number of ways to achieve this is given by $\binom{j}{i}$.
    \end{itemize}
Figure~\ref{fig:constructing a matching} provides an illustrative example of constructing a matching in $\Mat_{9,3}^{>2}(1)$ following the outlined steps. In the initial figure on the left, a matching in $\Mat_{5,4}$ is presented. Subsequently, in the middle figure, a vertex is added to the leftmost part, followed by the addition of another vertex which are then connected by a new arc, resulting in a matching in $\Mat_{9,3}^{>1}(1)$. Finally, another vertex is added, yielding a matching in $\Mat_{9,3}^{>2}(1)$. By counting the number of ways in each step, we have
\begin{align*}
\left|\Mat_{2p+j,p}^{>j}(i)\right| & = \binom{j}{i} \frac{(2p)!}{(2p-i)!} \binom{2p-i}{2i} (2p-2i-1)!!
= (2p-1)!! \binom{j}{i}  \binom{p}{i} 2^{i}.
\end{align*}
This proves the claim, thereby concluding \eqref{GUE Matching positive}.

\begin{figure}[h]
    \centering
\begin{tikzpicture}[scale = 0.5]
      \tikzset{enclosed/.style={draw, circle, inner sep=0pt, minimum size=.1cm, fill=black}}

    \draw [fill,red] (0,0) circle [radius=0.075] ;
    \draw [fill,red] (1,0) circle [radius=0.075] ;
    \draw [fill,red] (2,0) circle [radius=0.075] ;
    \draw [fill,red] (3,0) circle [radius=0.075] ;
    \draw [fill,red] (4,0) circle [radius=0.075] ;
    \draw [red] (0,0) to [out=55,in=125] (3,0);
    \draw [red] (1,0) to [out=55,in=125] (4,0);

    \draw [fill,red] (6,0) circle [radius=0.075] ;
    \draw [fill] (7,0) circle [radius=0.075] ;
    \draw [fill] (8,0) circle [radius=0.075] ;
    \draw [fill] (9,0) circle [radius=0.075] ;
    \draw [fill] (10,0) circle [radius=0.075] ;
    \draw [fill,red] (11,0) circle [radius=0.075] ;
    \draw [fill] (12,0) circle [radius=0.075] ;
    \draw [] (7,0) to [out=55,in=125] (10,0);
    \draw [] (8,0) to [out=55,in=125] (12,0);
    \draw [red] (6,0) to [out=55,in=125] (11,0);

    \draw [fill] (14,0) circle [radius=0.075] ;
    \draw [fill,red] (15,0) circle [radius=0.075] ;
    \draw [fill] (16,0) circle [radius=0.075] ;
    \draw [fill] (17,0) circle [radius=0.075] ;
    \draw [fill] (18,0) circle [radius=0.075] ;
    \draw [fill] (19,0) circle [radius=0.075] ;
    \draw [fill] (20,0) circle [radius=0.075] ;
    \draw [fill] (21,0) circle [radius=0.075] ;
    \draw [] (16,0) to [out=55,in=125] (19,0);
    \draw [] (17,0) to [out=55,in=125] (21,0);
    \draw [] (14,0) to [out=55,in=125] (20,0);

\end{tikzpicture}
    \caption{Constructing a matching in $\Mat_{9,3}^{>2}(1)$ step by step, newly added vertices and arcs in each step are colored in red.}
    \label{fig:constructing a matching}
\end{figure}
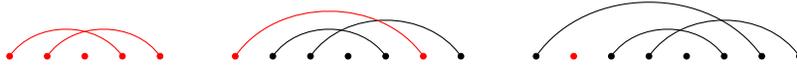

    

\medskip 

Next, we show \eqref{GUE Matching alternating}. For an integer $0\leq r < p$ and $j$, we enumerate the number $m_{j+2p,p,r}$ of matchings on $j+2p$ vertices with $p$ arcs, including additional markings on $r$ arcs connecting negative vertices.

\begin{itemize}
    \item Firstly, we select $2r$ vertices from the set of $j$ negative vertices and connect them with $r$ arcs.
    \smallskip
    \item Subsequently, we choose $2p-2r$ vertices from the remaining $2p+j-2r$ vertices and connect them with $p-r$ arcs.
\end{itemize}

This procedural gives
\begin{equation}
    m_{j+2p,p,r} = \binom{j}{2r} (2r-1)!! \binom{j+2p-2r}{2p-2r} (2p-2r-1)!! = (2p-1)!! \binom{j+2p-2r}{2p} \binom{p}{r}.
\end{equation}
Let $m_{j+2p,p,r}^{i}$ denote the number of matchings on $j+2p$ with $p$ arcs, where exactly $i$ arcs connect negative vertices, and $r$ of them are marked. We then establish the following relationship:
\begin{align}
\begin{split}
    \sum_{r=0}^p (-1)^r m_{j+2p,p,r} &= (2p-1)!! \sum_{r=0}^p \binom{j+2p-2r}{2p} \binom{p}{r} \\
    &= \sum_{r=0}^p \sum_{i\geq 0} (-1)^r m_{j+2p,p,r}^{i} \binom{i}{r}.
\end{split}
\end{align}
The identity \(\sum_{r=0}^p (-1)^r \binom{i}{r} = \delta_{i,0}\) implies \eqref{GUE Matching alternating}.
\end{proof}

\begin{rem}
The idea can be further extended to evaluate the case $j \not = l$ in general. 
Then it follows that 
\begin{align}
\frac{1}{(j+2m)!} \,  \mathcal{L}\Big( x^{2p} He_j(x) He_{j+2m}(x)  \Big)  & =   \sum_{l=0}^{p-m} \binom{j}{l}  \frac{(2p)!}{ (l+2m)! (p-l-m)! 2^{p-l-m}  },
\\
\frac{1}{(j+2m+1)!} \,  \mathcal{L}\Big( x^{2p+1} He_{j}(x) He_{j+2m+1} (x)  \Big) &=   \sum_{l=0}^{p-m} \binom{j}{l}  \frac{(2p+1)!}{ (l+2m+1)! (p-l-m)! 2^{p-l-m}  }.
\end{align}
See \cite[Proposition 13]{Fo21a} and references therein for an extensive earlier literature on such evaluations. 
\end{rem}

\subsection{Proof of Theorem~\ref{Thm_qGUE moments}}

We will prove Theorem~\ref{Thm_qGUE moments} (i) using the path counting method. We first describe the combinatorics of moments of $q$-Hermite polynomials. Let us recall some statistics for matchings. A \emph{crossing} is either
\begin{enumerate}
    \item a pair of arcs $(a,b)$ and $(c,d)$ with $a<c<b<d$, or 
    \smallskip 
    \item a pair consisting of an arc $(a,b)$ and an isolated vertex $c$ with $a<c<b$. 
\end{enumerate}
For a matching $M$, we denote the total number of crossings by $\cros(M)$. Pictorially, the first two diagrams in Figure~\ref{fig:crossings and nestings} illustrate two possible cases of crossings. For the second case, one might conceptualize an isolated vertex as connected to an imaginary arc reaching `infinity', thus forming a `crossing'.
A \emph{nesting} is either
\begin{enumerate}
    \item a pair of arcs $(a,b)$ and $(c,d)$ with $a<c<d<b$ or $c<a<b<d$, or
      \smallskip 
    \item a pair consisting of an arc $(a,b)$ and an isolated vertex $c$ with $c<a<b$. 
\end{enumerate} 
For a matching $M$, we denote the total number of nestings by $\nest(M)$. Pictorially, the last two diagrams in Figure~\ref{fig:crossings and nestings} illustrate two possible cases of nestings. For the second case, one might again think of an isolated vertex as connected to an imaginary arc extending to `infinity', thus forming a `nesting'.

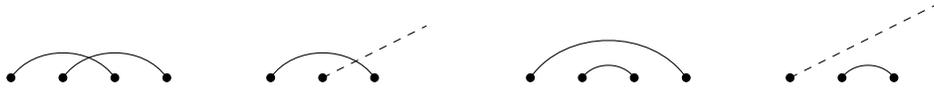
\begin{figure}[h]
    \centering
\begin{tikzpicture}[scale = 0.6905]
      \tikzset{enclosed/.style={draw, circle, inner sep=0pt, minimum size=.1cm, fill=black}}

    \draw [fill] (0,0) circle [radius=0.075] ;
    \draw [fill] (1,0) circle [radius=0.075] ;
    \draw [fill] (2,0) circle [radius=0.075] ;
    \draw [fill] (3,0) circle [radius=0.075] ;
    \draw (1,0) to [out=55,in=125] (3,0);
    \draw (0,0) to [out=55,in=125] (2,0);

    \draw [fill] (5,0) circle [radius=0.075] ;
    \draw [fill] (6,0) circle [radius=0.075] ;
    \draw [fill] (7,0) circle [radius=0.075] ;
    \draw (5,0) to [out=55,in=125] (7,0);
    \draw[dashed] (6,0) to (8,1);

    \draw [fill] (10,0) circle [radius=0.075] ;
    \draw [fill] (11,0) circle [radius=0.075] ;
    \draw [fill] (12,0) circle [radius=0.075] ;
    \draw [fill] (13,0) circle [radius=0.075] ;
    \draw (11,0) to [out=55,in=125] (12,0);
    \draw (10,0) to [out=55,in=125] (13,0);

    \draw [fill] (15,0) circle [radius=0.075] ;
    \draw [fill] (16,0) circle [radius=0.075] ;
    \draw [fill] (17,0) circle [radius=0.075] ;
    \draw (16,0) to [out=55,in=125] (17,0);
    \draw[dashed] (15,0) to (18,1.5);

\end{tikzpicture}
    \caption{Crossings and nestings}
    \label{fig:crossings and nestings}
\end{figure}
As a $q$-analogue of \eqref{eq: combinatorial interpretation of moments of Hermite}, the moments of $q$-Hermite polynomials, defined by the three-term recurrence in \eqref{3-term recurrence hat H}, admit a combinatorial interpretation in terms of matchings, crossings, and nestings as follows. 

\begin{prop}\label{prop: qGUE sum of matchings} 
Let $p,j$ be nonnegative integers. Then we have
\begin{equation}
\widehat{\mathfrak{m}}_{2p,j}^{ \rm (H) } = \sum_{\Mat_{2p+j,p}^{>j}} q^{\stat(M)},
\end{equation}
where $\stat(M)=\cros(M)+2\nest(M).$
\end{prop}
\begin{proof}

To analyze the statistic $\stat(M)$, let us begin with the following observation: for a matching $M$ and an arc $\alpha$ of $M$, let $\stat_\alpha(M)$ be defined as
\[
    \stat_\alpha(M):= C_1 + C_2 + 2(N_1 + N_2),
\]
where $C_1$ is the number of arcs that form a crossing with $\alpha$ whose opener is between the opener and closer of $\alpha$, $C_2$ is the number of isolated vertices that form a crossing with $\alpha$ which are less than the opener of $\alpha$, $N_1$ is the number of arcs that form a nesting with $\alpha$ whose opener is less than the opener of $\alpha$, and $N_2$ is the number of isolated vertices which are less than the opener of $\alpha$. Then, it is straightforward to observe that 
\begin{equation}\label{eq: observation}
    \stat(M) = \sum_{\alpha\in M}\stat_{\alpha}(M).
\end{equation}

Recalling the proof of Proposition~\ref{Prop_GUE moments}, we established a dynamic bijection between Hermite histories and matchings. This bijection starts with $j$ vertices and introduces a new vertex and an arc to the $m$-th isolated vertex for a South-East step labeled by $m$, or simply introduces a new vertex for a North-East step. Following this bijection, let us consider the case encountering a South-East step of height $k$ labeled by $m$. The corresponding matching can be described as follows: for a given matching with $k$ isolated vertices, we introduce a new vertex and connect it with the $m$-th isolated vertex. Let $\alpha$ denote this additional arc. After the final stage of this bijection to obtain a matching $M$, we have $\stat_{\alpha}(M)={(k-1)+(m-1)}$. By considering all $k$ possible labels (1 through $m$) to obtain a new arc $\alpha$, we have the additional statistic $q^{k-1}[k]_q$, which corresponds to the weight $\lambda_k$. When we encounter a North-East step, the weight is 1, and on the matching side, introducing a new vertex contributes 1 to $q^{\stat(M)}$. For both cases, the weight of the Hermite history for each step is recorded by $\stat_\alpha(M)$. This, with \eqref{eq: observation} completes the proof of the claim.
\end{proof}

\begin{rem} For the other version of $q$-Hermite polynomial $\widetilde{H}_n(x;q)$, given in \eqref{3-term recurrence tilde H}, we have
\begin{equation}
\widetilde{\mathfrak{m}}_{2p,j}^{ \rm (H) } = \sum_{\Mat_{2p+j,p}^{>j}} q^{\cros(M)}.
\end{equation}    
\end{rem}

We are now prepared to establish one of the key theorems in this paper, Theorem~\ref{Thm_qGUE moments} (i), which provides a positive formula for the spectral moments of the $q$-deformed GUE.

\begin{proof}[Proof of Theorem~\ref{Thm_qGUE moments}~\textup{(i)}]
According to Proposition~\ref{prop: qGUE sum of matchings}, it is sufficient to establish the following equality:
\begin{align} \label{eq: positive formula for L(x^2p H_j^2)}
\sum_{M\in\Mat_{2p+j,p}^{>j}} q^{\cros(M)+2\nest(M)}
= \sum_{i\ge 0} q^{(j-i)(2p-i)+i(i-1)/2}\dfrac{[2p]_q!}{[2p-2i]_q!![i]_q!}\discretebinom{j}{i}_q.
\end{align}
To achieve this, we extend the idea of proof employed in Proposition~\ref{Prop_GUE moments}, the $q=1$ case of Theorem~\ref{Thm_qGUE moments} given in Section~\ref{subsec: spectral moments of the GUE revisited}.

We first refine \eqref{eq: positive formula for L(x^2p H_j^2)} as follows:
\begin{equation}\label{eq: refined identity Mat>j(i)}
    \sum_{M\in\Mat_{2p+j,p}^{>j}(i)} q^{\cros(M)+2\nest(M)} = q^{(j-i)(2p-i)+i(i-1)/2}\dfrac{[2p]_q![j]_q!}{[2p-2i]_q!! [i]_q! [j-2i]_q!},
\end{equation}
where $\Mat_{2p+j,p}^{>j}(i)$ denotes the set of matchings in $\Mat_{2p+j,p}^{>j}$ with $i$ arcs connected to the first $j$ vertices, as in the proof of Proposition~\ref{Prop_GUE moments}. By summing both sides of \eqref{eq: refined identity Mat>j(i)} over $1\le i \le j$, we establish \eqref{eq: positive formula for L(x^2p H_j^2)}. Therefore, it remains to validate \eqref{eq: refined identity Mat>j(i)}.

We start by claiming that
\[
    \sum_{M\in\Mat_{a,b}} q^{\cros(M)+2\nest(M)} = \discretebinom{a}{b}_q [2b-1]_q!!.
\]
We establish this assertion through induction on $a$ and $b$. Let $\Mat_{a,b}$ be partitioned as
\[
    \Mat_{a,b} = \Mat_{a,b}^{(1)} \cup \bigcup_{x>1} \Mat_{a,b}^{(x)}
\]
Here, we let $\Mat_{a,b}^{(x)}$ be the set of matchings in $\Mat_{a,b}$ where the first vertex $v=1$ forms an arc with the vertex $x$. Notably, vertex $v$ is an isolated vertex in $\Mat_{a,b}^{(1)}$.

For a matching $M\in\Mat_{a,b}^{(1)}$, since $v$ is isolated and there are $b$ arcs after vertex $v$, vertex $v$ forms $b$ nestings and thus contributes $2b$ to $\stat(M)$. Additionally, deleting the vertex $v$ gives a bijection between $\Mat_{a,b}^{(1)}$ and $\Mat_{a-1,b}$. By the induction hypothesis, we obtain
\begin{equation}\label{eq: sum over 1 is isolated}
    \sum_{M\in\Mat_{a,b}^{(1)}}q^{\stat(M)} = q^{2b}\discretebinom{a-1}{b}[2b-1]_q!!.
\end{equation}
On the other hand, for a matching $M\in\Mat_{a,b}^{(x)}$, the first arc $(1,x)$ contributes $x-1$ to $\stat(M)$. Furthermore, deleting the arc $(1,x)$ gives a bijection between $\Mat_{a,b}^{(x)}$ and $\Mat_{a-2,b-1}$. By the induction hypothesis, we obtain
\begin{equation}\label{eq: sum over 1 is connected}
    \sum_{x=2}^{a}\sum_{M\in\Mat_{a,b}^{(x)}}q^{\stat(M)} = [a-1]_q \discretebinom{a-2}{2b-2}[2b-3]_q!!.
\end{equation}
Combining equations \eqref{eq: sum over 1 is isolated} and \eqref{eq: sum over 1 is connected}, we get
    \begin{align*}
        \sum_{M\in\Mat_{a,b}}q^{\stat(M)} &= \sum_{x=1}^{a}\sum_{M\in\Mat_{a,b}^{(x)}}q^{\stat(M)}\\
        &=q^{2b}\discretebinom{a-1}{b}[2b-1]_q!! + [a-1]_q \discretebinom{a-2}{2b-2}[2b-3]_q!!\\
        &=q^{2b}\discretebinom{a-1}{b}[2b-1]_q!! + \discretebinom{a-1}{2b-1}[2b-1]_q!! =\discretebinom{a}{b}[2b-1]_q!!.
    \end{align*}
In the third equation, we use the identity $[n]_q\discretebinom{n-1}{k}_q=[n-k]_q\discretebinom{n}{k}_q$, and in the last equation, we use the $q$-Pascal identity. In particular, by taking $a=2p-i$ and $b=p-i$, we have
\begin{equation}
    \sum_{M\in\Mat_{2p-i,p-i}} q^{\stat(M)} = \discretebinom{2p-i}{i}[2p-2i-1]_q!!.
\end{equation}

We proceed to prove a special case of \eqref{eq: refined identity Mat>j(i)} where $j=i$,
\begin{equation}\label{eq: Mat_2p+i,p,i^<j}
    \sum_{M\in\Mat_{2p+i,p,i}^{<i}(i)} q^{\stat(M)} = q^{i(i-1)/2}\dfrac{[2p]_q!}{[2p-i]_q!} \discretebinom{2p-i}{2p-2i}[2p-2i-1]_q!!.
\end{equation}
To prove this, fix a matching $M \in \Mat_{2p-i,p-i}$. We introduce $i$ new vertices to the left of $M$, placing the resulting matching in $\Mat_{2p,p-i}^{<i}(0)$. Then, we introduce new arcs and their closers larger than $i$ one by one by connecting vertex $1$ first, followed by vertex $2$, and so forth. When introducing the first arc, subject to the conditions mentioned above, there are $(2p-i+1)$ choices. These choices contribute a factor $q^{i-1}[2p-i+1]_q$. Similarly, when introducing the second arc, there are $(2p-i+2)$ choices, resulting in another factor $q^{i-2}[2p-i+2]_q$. By inductively applying this argument, we establish the claim.

Finally, consider a projection map 
\[
    \mathcal{P}:\Mat_{2p+j,p}^{<j}(i)\rightarrow \Mat_{2p+i,p}^{<i}(i)
\]
which is defined by deleting $j-i$ isolated vertices less than $j+1$. Suppose we fix a matching $M\in\Mat_{2p+i,p}^{<i}(i)$. Then, there are $\binom{j}{i}$ choices for introducing $j-i$ new isolated vertices to $M$ to obtain a matching in $\Mat_{2p+j,p}^{<j}(i)$. For each such choice, we observe that the additional contribution of an isolated vertex $v$ to $q^{\stat(M)}$ is $q^{(2p-i) + N}$, where $N$ is the count of chosen vertices less than $j+1$ larger than $v$. Utilizing the standard combinatorial interpretation of $q$-binomial coefficients, we conclude that
\[
    \sum_{\tilde{M}\in\mathcal{P}^{-1}(M)}q^{\stat(\tilde{M})} = q^{\stat(M)+(j-1)(2p-i)}\discretebinom{j}{i}.
\]
Combining this with \eqref{eq: Mat_2p+i,p,i^<j}, we have
\begin{align*}
    \sum_{M\in\Mat_{2p+j,p}^{>j}(i)} q^{\stat(M)} &= q^{(j-1)(2p-i)+i(i-1)/2}\dfrac{[2p]_q!}{[2p-i]_q!} \discretebinom{2p-i}{2p-2i}[2p-2i-1]_q!!\discretebinom{j}{i}\\&= q^{(j-i)(2p-i)+i(i-1)/2}\dfrac{[2p]_q![j]_q!}{[2p-2i]_q!! [i]_q! [j-2i]_q!}.
\end{align*}
\end{proof}

\begin{rem}
\label{Lem: Mat 2p+j,p with r}

We expect that the alternating series formula \eqref{qGUE Matching alternating} also has a combinatorial proof. 
Let us rewrite \eqref{qGUE Matching alternating} as
\begin{equation}\label{eq: sum over Mat 2p+j,p >j}
\sum_{M\in\Mat_{2p+j,p}^{>j}} q^{\stat(M)} =  \sum_{ r=0 }^{ p } (-1)^{ r }  q^{ r ( r - 1  ) }  \discretebinom{j}{2r}_q [2r-1]_q!!    \discretebinom{j+2p-2r}{2p-2r}_q   [2p-2r-1]_q!!.
\end{equation}
We introduce further notation to refine the above identity. For a matching $m\in \Mat_{2p+j,p}$, let $\cl^{\le j}(m)$ denote the count of closers less than or equal to $j$. Then, for $p,j\ge 0$ and $r\le p, \frac{j}{2}$, we expect that 
\begin{equation}\label{eq: Mat 2p+j,p with r}
    \sum_{M\in\Mat_{2p+j,p}} q^{\stat(M)}\discretebinom{\cl^{\le j}(m)}{r}_{q^2} =  \discretebinom{j}{2r}_q [2r-1]_q!!    \discretebinom{j+2p-2r}{2p-2r}_q   [2p-2r-1]_q!!.
\end{equation} 
This identity can be verified at small order using computer algebra. It is worth noting that \eqref{eq: sum over Mat 2p+j,p >j} follows from \eqref{eq: Mat 2p+j,p with r}. 
For this purpose, let $A_{p,j,k}$ be the generating function of $q^{\cros(M)+2\nest(M)}$ over all matchings $m\in \Mat_{2p+j,p}$ with exactly $k$ closers are less or equal to $j$. Note that $A_{p,j,0}$ is equal to the left-hand side of \eqref{eq: sum over Mat 2p+j,p >j}. In addition, the left-hand side of \eqref{eq: Mat 2p+j,p with r}, denoted by $a_{p,j,r}$, is given by
\[
    a_{p,j,r} = \sum_{k\ge r} A_{p,j,k} \discretebinom{k}{r}_{q^2}.
\]
Using the identity $\sum_{r=0}^k (-1)^r q^{r(r-1)} \discretebinom{k}{r}_{q^2}=0$ for $r\ge 1$, we have
\[
    A_{p,j,0} = \sum_{r\ge 0} (-1)^r q^{r(r-1)} a_{p,j,r},
\]
which leads to  \eqref{eq: sum over Mat 2p+j,p >j}.
\end{rem}



\section{Scaled moments and density}

\subsection{Proof of Theorem~\ref{Thm_genus one}}

In this subsection, we prove Theorem~\ref{Thm_genus one}.

Let us write  
\begin{equation} \label{def of FN(l)}
F_N(l):=q^{ p-l(2p-l)+l(l-1)/2  }  (1-q)^p  \dfrac{[2p]_q!}{[2p-2l]_q!! \, [l]_q!}  \sum_{j=0}^{N-1}  q^{j(2p-l)}  \discretebinom{j}{l}_q .
\end{equation}
Then we have
\begin{equation} \label{m hat F_N sum}
\Big(q(1-q)\Big)^p \widehat{m}_{N,2p}^{ \rm (dH) }=  \sum_{l=0}^p F_N(l). 
\end{equation}
We shall verify that for each $l$, $F_N(l)$ has a $1/N$-expansion with a leading order of $O(N^{l+1-p})$. 
Hence, to prove the theorem, the specific large-$N$ asymptotic behaviours of the three terms $F_N(p)$, $F_N(p-1)$, and $F_N(p-2)$ are required.
In particular, one needs to precisely analyze the summation in \eqref{def of FN(l)}, for which we will implement some ideas from recent works \cite{BKS23,ACC23}.

Recall that $s(n,k)$ denotes the Stirling numbers \eqref{def of Stirling number}. 
We first provide the asymptotic behaviours of the pre-factors in \eqref{def of FN(l)}.

\begin{lem} \label{Lem_genus 1 prefactor asymp}
Let $q$ be scaled as in \eqref{q lambda scaling}. Then as $N \to \infty,$ we have 
\begin{align*}
&\quad (1-q)^p  \dfrac{[2p]_q!}{[2p-2l]_q!! \, [l]_q!}  =  \Big( \frac{\lambda}{N} \Big)^{p} \frac{(2p)!}{(2p-2l)!!\,l!} 
 \bigg( 1+  \frac{ 3 l^2 - 4 l p - 2 p^2-l}{4} \frac{\lambda}{N} + \mathfrak{a}(p,l)\frac{\lambda^2}{N^2} +O(\frac{1}{N^3}) \bigg),
\end{align*}
where 
\begin{align*}
\mathfrak{a}(p,l)&= \frac{s(2p+1,2p-1)}{4} -s(p-l+1,p-l-1) - \frac{s(l+1,l-1)}{4}
\\
&\quad + \frac{1}{144} (12 l - 33 l^2 - 30 l^3 + 63 l^4 - 8 p + 60 l p + 
   48 l^2 p - 180 l^3 p + 6 l p^2 + 126 l^2 p^2 - 4 p^3 - 36 p^4 ). 
\end{align*}
\end{lem}

\begin{proof}
Note that by definition, 
\begin{align*}
(1-q)^p  \dfrac{[2p]_q!}{[2p-2l]_q!! \, [l]_q!} &= \frac{ \prod_{j=1}^{2p} (1-q^j)   }{  \prod_{j=1}^{p-l} (1-q^{2j})  \prod_{j=1}^{l} (1-q^j)  } . 
\end{align*}
Since 
$$
1-q^m = 1-e^{ -m\lambda/N  } =\frac{m\lambda}{N} \sum_{r=0}^\infty \frac{(-1)^{r}}{(r+1)!} \Big( \frac{m\lambda}{N} \Big)^{r},  
$$
we have 
\begin{align*}
&\quad \prod_{j=1}^{2p} (1-q^j)  = \Big( \frac{\lambda}{N} \Big)^{2p} (2p)! \prod_{m=1}^{2p}\bigg( 1- \frac{ m\lambda }{ 2N } + \frac{(m\lambda)^2}{ 6N^2 } +O(\frac{1}{N^3}) \bigg)
\\
&=  \Big( \frac{\lambda}{N} \Big)^{2p} (2p)! \bigg( 1- \frac{p(2p+1)}{2} \frac{\lambda}{N} + \Big( \frac{s(2p+1,2p-1)}{4}+ \frac{p(2p+1)(4p+1)}{18}  \Big)  \frac{\lambda^2}{N^2} +O(\frac{1}{N^3}) \bigg),
\end{align*}
\begin{align*}
 \prod_{j=1}^{p-l} (1-q^{2j}) & =  \Big( \frac{\lambda}{N} \Big)^{p-l} (2p-2l)!!    \bigg( 1- \frac{(p-l)(p-l+1)}{2}  \frac{ \lambda}{N} 
\\
& \quad + \Big(  s(p-l+1,p-l-1) + \frac{(p-l)(p-l+1)(2p-2l+1)}{9}  \Big)  \frac{\lambda^2}{N^2} +O(\frac{1}{N^3}) \bigg),
\end{align*}
and
\begin{align*}
 \prod_{j=1}^{l} (1-q^j)  =   \Big( \frac{\lambda}{N} \Big)^{l} l! \bigg( 1- \frac{l(l+1)}{4} \frac{\lambda}{N} + \Big( \frac{s(l+1,l-1)}{4}+ \frac{l(l+1)(2l+1)}{36}  \Big)  \frac{\lambda^2}{N^2} +O(\frac{1}{N^3}) \bigg).
\end{align*}
This completes the proof. 
\end{proof}

We now begin with analyzing the summation in \eqref{def of FN(l)}.
For this, let 
\begin{equation}
f_l(j)= q^{j(2p-l)}  \discretebinom{j}{l}_q. 
\end{equation}

\begin{lem} \label{Lem_asymp fl Nt}
Let $q$ be scaled as in \eqref{q lambda scaling}. Then as $N \to \infty,$ we have 
\begin{align}
f_l(Nt)=  \frac{N^l}{ l!\,\lambda^l } e^{ -(2p-l) \lambda t  } (1-e^{-\lambda t})^l \Bigg( 1+ \mathcal{B}_{l,1}(t) \frac{\lambda}{N} + \mathcal{B}_{l,2}(t) \frac{\lambda^2}{N^2} +O(N^{-3}) \bigg), 
\end{align}
where 
\begin{align}
\mathcal{B}_{l,1}(t) &= \frac{l(l+1)}{4} - \frac{ e^{-\lambda t} }{1-e^{ -\lambda t }} \frac{l(l-1)}{2},
\\
\begin{split}
\mathcal{B}_{l,2}(t) &=  \Big(\frac{ e^{-\lambda t} }{1-e^{ -\lambda t }}\Big)^2  s(l,l-2)  - \frac{ e^{-\lambda t} }{1-e^{ -\lambda t }} \frac{ (-1 + l) l (-2 + 7 l + 3 l^2) }{24} 
\\
&\quad -\frac{s(l+1,l-1)}{4} + \frac{l (1 + l) (-4 + l + 9 l^2)}{144} .
\end{split}
\end{align}
Furthermore, we have 
\begin{align}
N \int_0^1 f_l(Nt)\,dt = \Big( \frac{N}{\lambda}\Big)^{l+1} \bigg( \mathcal{A}_{l,0}  +  \mathcal{A}_{l,1} \Big( \frac{N}{\lambda}\Big)^{-1} + \mathcal{A}_{l,2} \Big( \frac{N}{\lambda}\Big)^{-2} +O(N^{-3}) \bigg),
\end{align}
where 
\begin{equation} \label{def of mathcal A l0}
\mathcal{A}_{l,0} = \frac{  (2p-l-1)!}{ (2p)! }  I_{1-e^{-\lambda}} (l+1,2p-l),
\end{equation}
\begin{equation}  \label{def of mathcal A l1}
\mathcal{A}_{l,1} =  \frac{l(l+1)}{4}    \frac{  (2p-l-1)!}{ (2p)! }  I_{1-e^{-\lambda}} (l+1,2p-l) - \frac{ (l-1)}{2} \frac{  (2p-l)!}{ (2p)! }  I_{1-e^{-\lambda}} (l,2p-l+1) ,     
\end{equation}
and
\begin{align} 
\begin{split}  \label{def of mathcal A l2}
\mathcal{A}_{l,2} &= \Big(  -\frac{s(l+1,l-1)}{4} + \frac{l (1 + l) (-4 + l + 9 l^2)}{144} \Big)    \frac{  (2p-l-1)!}{ (2p)! }  I_{1-e^{-\lambda}} (l+1,2p-l) 
\\
&\quad - \frac{ ( l-1)   (-2 + 7 l + 3 l^2) }{24}    \frac{  (2p-l)!}{ (2p)! }  I_{1-e^{-\lambda}} (l,2p-l+1) 
\\
&\quad +s(l,l-2)  \frac{1}{l(l-1)}   \frac{  (2p-l+1)!}{ (2p)! }  I_{1-e^{-\lambda}} (l-1,2p-l+2).   
\end{split}
\end{align}
\end{lem}

\begin{proof}

With the scaling \eqref{q lambda scaling} and with $j=t N$,  
\begin{align*}
 \discretebinom{j}{l}_q &= \frac{[j]_q \,!}{ [l]_q !\,[j-l]_q! } = \frac{ (1-q^j) \dots (1-q^{j-l+1}) }{  (1-q^l)\dots (1-q) } 
 = \frac{N^l}{ l!\,\lambda^l } (1-e^{-\lambda t})^l \Big(1+O(N^{-1}) \Big), 
\end{align*}
as $N\to \infty.$
Here, 
\begin{align*}
&\quad (1-q^j) \dots (1-q^{j-l+1})  = (1-e^{-\lambda t})^l \prod_{m=0}^{l-1} \Big( 1 - \frac{ e^{-\lambda t} }{1-e^{ -\lambda t }} \frac{m \lambda }{N}  - \frac{ e^{-\lambda t} }{1-e^{ -\lambda t }} \frac{m^2 \lambda^2 }{2N^2} +O(N^{-3}) \Big)  
\\
&= (1-e^{-\lambda t})^l \bigg( 1- \frac{ e^{-\lambda t} }{1-e^{ -\lambda t }} \frac{l(l-1)}{2} \frac{ \lambda }{N} +\frac{ e^{-\lambda t} }{1-e^{ -\lambda t }}  \Big(  \frac{ e^{-\lambda t} }{1-e^{ -\lambda t }}  s(l,l-2)-\frac{ (l-1) l ( 2 l-1) }{12}   \Big)   \frac{\lambda^2}{N^2} +O(N^{-3}) \bigg) 
\end{align*}
Since 
\begin{align*}
&\quad \Big( (1-q^l) (1-q^{l-1}) \dots (1-q) \Big)^{-1}
 \\
 &=    \frac{N^l}{l!\,\lambda^l}  \bigg( 1 + \frac{l(l+1)}{4} \frac{\lambda}{N} + \Big( -\frac{s(l+1,l-1)}{4} +\frac{l^2(l+1)^2}{16} - \frac{l(l+1)(2l+1)}{36}  \Big)  \frac{\lambda^2}{N^2} +O(N^{-3}) \bigg)
\end{align*} 
we have
\begin{align*}
&\quad \discretebinom{j}{l}_q \Big( \frac{N^l}{ l!\,\lambda^l } (1-e^{-\lambda t})^l \Big)^{-1}
\\
&= 1+ \bigg(  \frac{l(l+1)}{4} - \frac{ e^{-\lambda t} }{1-e^{ -\lambda t }} \frac{l(l-1)}{2} \bigg) \frac{\lambda}{N}  +\bigg(  \Big(\frac{ e^{-\lambda t} }{1-e^{ -\lambda t }}\Big)^2  s(l,l-2) 
\\
&\qquad - \frac{ e^{-\lambda t} }{1-e^{ -\lambda t }} \frac{ (-1 + l) l (-2 + 7 l + 3 l^2) }{24} -\frac{s(l+1,l-1)}{4} + \frac{l (1 + l) (-4 + l + 9 l^2)}{144} \bigg) \frac{\lambda^2}{N^2}+O(N^{-3}),
\end{align*}
which completes the proof of the first assertion. The second assertion follows from straightforward computations using 
\begin{align*}
&\quad \int_0^1 e^{ - \lambda t\,a  } (1-e^{-\lambda t})^b \,dt =\frac{1}{\lambda} \int_{0}^{ 1-e^{-\lambda} } s^b (1-s)^{a-1} \,ds
\\
&=  \frac{1}{\lambda} B_{1-e^{-\lambda}} (b+1,a)
= \frac{ 1 }{\lambda} \frac{ b! \,(a-1)!}{ (a+b)! }  I_{1-e^{-\lambda}} (b+1,a). 
\end{align*}
\end{proof}

Next, we derive the asymptotic behaviour of $\sum_{j=0}^{N-1} f(j) $.

\begin{lem}
Let $q$ be scaled as in \eqref{q lambda scaling}. Then as $N \to \infty,$ we have 
\begin{align}
\sum_{j=0}^{N-1} f_l(j)& = \mathcal{C}_{l,0} \Big( \frac{N}{\lambda}\Big)^{l+1} + \mathcal{C}_{l,1} \Big( \frac{N}{\lambda}\Big)^{l} +\mathcal{C}_{l,2} \Big( \frac{N}{\lambda}\Big)^{l-1}  +  O\Big( \frac{N}{\lambda}\Big)^{l-2},     
\end{align}
where        
\begin{align}
 \mathcal{C}_{l,0} &=  \mathcal{A}_{l,0} , 
 \\
  \mathcal{C}_{l,1} &=   \mathcal{A}_{l,1}  -\frac12  \frac{1}{ l!  } e^{ -(2p-l) \lambda    } (1-e^{-\lambda  })^l ,
  \\
  \mathcal{C}_{l,2} &=  \mathcal{A}_{l,2}  +\frac{1}{12} \frac{ 1 }{ l! } e^{ -(2p-l) \lambda   } (1-e^{-\lambda })^l \, \Big(  -\frac{3l^2+l+4p}{2} +(3l^2-2l) \frac{ e^{-\lambda } }{ 1-e^{-\lambda } } \Big). 
\end{align}
Here $\mathcal{A}_{l,0}$, $\mathcal{A}_{l,1}$ and $\mathcal{A}_{l,0}$ are given by \eqref{def of mathcal A l0}, \eqref{def of mathcal A l1} and \eqref{def of mathcal A l2}, respectively. 
\end{lem}

\begin{proof}
We apply the Euler–Maclaurin formula \cite[Section 2.10]{NIST} 
\begin{equation}
\begin{split} \label{EMF}
\sum_{j=m}^{n} f(j) = \int_{m}^{n} f(x) \, dx + \frac{f(m)+f(n)}{2} + \sum_{k=1}^{ l-1 } \frac{B_{2k}}{(2k)!}\Big(f^{(2k-1)}(n)-f^{(2k-1)}(m)\Big) + R_l,
\end{split}
\end{equation}
where $B_n$ is the Bernoulli number defined by the generating function
\begin{equation} \label{Bernoulli number}
\frac{t}{e^t-1}=\sum_{n=0}^\infty B_n \frac{t^n}{n!} .
\end{equation}
Here the error term $R_l$ satisfies the estimate 
\[
|R_l| \le \frac{4\,\zeta(2l)}{(2\pi)^{2l}} \int_m^n|f^{(2l)}(x)|\,dx,
\]
where $\zeta$ is the Riemann zeta function.
By using \eqref{EMF}, we have 
\begin{align*}
\sum_{j=0}^{N-1} f_l(j) & = -f_l(N) + \sum_{j=0}^{N} f_l(j)
\\
&= \int_0^N f_l(x) \,dx + \frac{f_l(0)-f_l(N)}{2} + \sum_{k=1}^{ m-1 } \frac{B_{2k}}{(2k)!}\Big(f_l^{(2k-1)}(N)-f_l^{(2k-1)}(0)\Big) + R_m. 
\end{align*}
Note that 
\begin{align*}
\int_0^N f(x) \,dx =N \int_0^1 q^{ 2N t (2p-l) } \discretebinom{Nt}{l}_q\,dt .
\end{align*} 
Furthermore, as a consequence of Lemma~\ref{Lem_asymp fl Nt}, we have 
\begin{align*}
f_l(N) =  \frac{N^l}{ l!\,\lambda^l } e^{ -(2p-l) \lambda    } (1-e^{-\lambda  })^l \Bigg( 1+   \Big( \frac{l(l+1)}{4} - \frac{ e^{-\lambda } }{1-e^{ -\lambda  }} \frac{l(l-1)}{2} \Big) \frac{\lambda}{N} +  +O(N^{-2}) \bigg)
\end{align*}
and 
\begin{align*}
f_l'(N) =  \frac{N^l}{ l!\,\lambda^l } e^{ -(2p-l) \lambda   } (1-e^{-\lambda })^l \, \frac{ \lambda }{N} \Big( -(2p-l)  +l \frac{ e^{-\lambda } }{ 1-e^{-\lambda } } \Big). 
\end{align*}
Combining all of the above with $B_2=1/6$, the lemma follows. 
\end{proof}

We are now ready to prove Theorem~\ref{Thm_genus one}. 

\begin{proof}[Proof of Theorem~\ref{Thm_genus one}]

By Lemma~\ref{Lem_genus 1 prefactor asymp}, we have 
\begin{align*}
 &\quad q^{p-l(2p-l)+l(l-1)/2} (1-q)^p \dfrac{[2p]_q!}{[2p-2l]_q!! \, [l]_q!}   \bigg( \Big( \frac{\lambda}{N} \Big)^{p} \frac{(2p)!}{(2p-2l)!!\,l!}  \bigg)^{-1} \bigg|_{l=p}
 \\
 &=  1 - \frac{ p(p+3) }{4} \frac{\lambda}{N} +\bigg( \frac{p (4 + 63 p + 20 p^2 - 63 p^3)}{144} +\frac{s(2p+1,2p-1)}{4}  - \frac{s(p+1,p-1)}{4} \bigg) \frac{\lambda^2}{N^2} +O(N^{-3}),
\end{align*}
and 
\begin{align*}
 q^{p-l(2p-l)+l(l-1)/2} (1-q)^p \dfrac{[2p]_q!}{[2p-2l]_q!! \, [l]_q!}   \bigg( \Big( \frac{\lambda}{N} \Big)^{p} \frac{(2p)!}{(2p-2l)!!\,l!}  \bigg)^{-1} \bigg|_{l=p-1}
=  1 - \frac{ p^2+p+4 }{4} \frac{\lambda}{N} + O(N^{-2}).
\end{align*}
Therefore by \eqref{def of FN(l)}, we have
\begin{align*}
F_N(p) & = \frac{N}{\lambda}  \frac{(2p)!}{p!}  \mathcal{C}_{p,0} +  \frac{(2p)!}{p!}  \bigg[ \mathcal{C}_{p,1} - \frac{p(p+3)}{4} \mathcal{C}_{p,0} \bigg]
\\
& \quad  + \Big( \frac{N}{\lambda}\Big)^{-1} \frac{(2p)!}{p!}  \bigg[  \mathcal{C}_{p,2}   - \frac{p(p+3)}{4} \mathcal{C}_{p,1} 
\\
&\qquad +\bigg( \frac{p (4 + 63 p + 20 p^2 - 63 p^3)}{144} +\frac{s(2p+1,2p-1)}{4}  - \frac{s(p+1,p-1)}{4} \bigg) \mathcal{C}_{p,0} \bigg]  +O(N^{-2}). 
\end{align*}
Also we have 
\begin{align*}
F_N(p-1) &=  \frac12 \frac{(2p)!}{(p-1)!}  \mathcal{C}_{p-1,0} 
 + \Big( \frac{N}{\lambda}\Big)^{-1}  \frac12 \frac{(2p)!}{(p-1)!} \bigg[   \mathcal{C}_{p-1,1} - \frac{p^2+p+4}{4} \mathcal{C}_{p-1,0}  \bigg] +O(N^{-2}),
 \\
F_N(p-2) &=  \Big( \frac{N}{\lambda}\Big)^{-1} \frac{1}{8} \frac{(2p)!}{ (p-2)! } \mathcal{C}_{p-2,0} +O(N^{-2}).  
\end{align*}
Combining these asymptotic behaviours, it follows that 
\begin{align} 
F_N(p)+ F_N(p-1) +F_N(p-2) = \mathcal{D}_0 \,\frac{N}{\lambda} +\mathcal{D}_1 + \mathcal{D}_2 \frac{\lambda}{N} +O(N^{-2}), 
\end{align}
where 
\begin{align*}
\mathcal{D}_0=     \frac{(2p)!}{p!}  \mathcal{C}_{p,0}, \qquad  \mathcal{D}_1=    \frac{(2p)!}{p!} \bigg[ \mathcal{C}_{p,1} - \frac{p(p+3)}{4} \mathcal{C}_{p,0} \bigg]  +  \frac12 \frac{(2p)!}{(p-1)!}  \mathcal{C}_{p-1,0}
\end{align*} 
and 
\begin{align*}
\begin{split}
\mathcal{D}_2 &= \frac{(2p)!}{p!}  \bigg[  \mathcal{C}_{p,2}   - \frac{p(p+3)}{4} \mathcal{C}_{p,1}  +\bigg( \frac{p (4 + 63 p + 20 p^2 - 63 p^3)}{144} +\frac{s(2p+1,2p-1)}{4}  - \frac{s(p+1,p-1)}{4} \bigg) \mathcal{C}_{p,0} \bigg] 
\\
&\quad +  \frac12 \frac{(2p)!}{(p-1)!} \bigg[   \mathcal{C}_{p-1,1} - \frac{p^2+p+4}{4} \mathcal{C}_{p-1,0}  \bigg]  +  \frac{1}{8} \frac{(2p)!}{ (p-2)! } \mathcal{C}_{p-2,0} .
\end{split}
\end{align*}

Write $e^{-\lambda}=x.$ By \cite[Eq.(8.17.18)]{NIST}, we have 
\begin{align*}
I_{1-x}(p,p+1) &= I_{1-x} (p+1,p) + \frac{(2p)!}{ (p!)^2}  x^p(1-x)^p,
\\
I_{1-x}(p-1,p+2) &= I_{1-x}(p+1,p) + \frac{(2p)!}{(p-1)!(p+1)!} x^{p+1}(1-x)^{p-1} +\frac{(2p)!}{ (p!)^2}  x^p(1-x)^p.
\end{align*}
Using these, we have 
\begin{align*}
\mathcal{C}_{p,0} &= \frac{(p-1)!}{(2p)!} I_{1-x}(p+1,p),
\\
\mathcal{C}_{p-1,0} &
= \frac{  p!}{ (2p)! }  I_{1-x} (p+1,p) + \frac{ x^p(1-x)^p}{p!},
\\
\mathcal{C}_{p-2,0} &
=\frac{  (p+1)!}{ (2p)! }  I_{1-x } (p+1,p)  + \frac{ x^{p+1}(1-x)^{p-1} }{ (p-1)!} +\frac{p+1}{p!} x^p(1-x)^p     
\end{align*}
and 
\begin{align*}
\mathcal{C}_{p,1} 
&= \frac{-p+3}{4} \frac{p!}{(2p)!} I_{1-x}(p+1,p) -\frac12 \frac{1}{(p-1)!} x^p(1-x)^p ,
\\
\mathcal{C}_{p-1,1}
&= \frac{-p^2+p+4}{4}  \frac{ p!}{ (2p)! }  I_{1-x}(p+1,p) + \frac{-p^2+p+4}{4} \frac{x^p(1-x)^p}{ p! } - \frac{1}{2(p-2)!} x^{p+1}(1-x)^{p-1}. 
\end{align*}
Also we have 
\begin{align*}
\mathcal{C}_{p,2} &= \Big(  -\frac{s(p+1,p-1)}{4} + \frac{p+1}{p-1}s(p,p-2) -\frac{ p (16 - 51 p + 14 p^2 + 9 p^3)}{144}   \Big)      \frac{  (p-1)!}{ (2p)! }  I_{1-x } (p+1, p) 
\\
&\quad +s(p,p-2)    \Big(   \frac{1}{(p-1) \, p!} x^{p+1}(1-x)^{p-1} + \frac{ (p+1) }{ p(p-1) } \frac{1}{ p!}  x^p(1-x)^p \Big)
\\
&\quad +  \frac{-2 + 4 p - 7 p^2 - 3 p^3} {24} \frac{ 1 }{ p! } x^p (1-x)^p   +   \frac{1}{12} \frac{ 1 }{ p! } (3p^2-2p) x^{p+1}(1-x)^{p-1}. 
\end{align*} 
Using these, after some computations, we obtain 
\begin{equation}
\mathcal{D}_1 =0, \qquad \mathcal{D}_2= -\frac{p}{6} I_{1-x}(p+1,p) -\frac{(2p-1)!}{6(p-1)!^2} x^p(1-x)^{p-1} \Big( 2+p-(2p+1)x \Big).
\end{equation}
Here, we have used 
\begin{equation}
s(n,n-2)= \frac{(n-2) (n-1) n (3 n-1)}{24}. 
\end{equation}
With the same strategy, one can also demonstrate that the next $O(N^{-2})$-term vanishes, and we omit the details.
This shows the theorem. 
\end{proof}

\subsection{Proof of Theorem~\ref{Thm_genus one density}}

We first show the leading order of Theorem~\ref{Thm_genus one density}.

\begin{prop}[\textbf{Limiting spectral density for the $q$-deformed GUE}] \label{Prop_density qGUE}
Let 
\begin{equation}
a(\lambda)=   \frac{2\sqrt{(1-e^{-\lambda}) e^{-\lambda}} }{ \sqrt{\lambda } } 
\end{equation} 
Then for $|x|<1/\sqrt{\lambda}$, we have 
\begin{equation}
\frac{1}{\sqrt{N}} \widehat{\rho}_N^{\rm (dH)}( \sqrt{N} x  )  \to \widehat{\rho}^{\rm (dH)}(x), 
\end{equation}
where $\widehat{\rho}^{\rm (dH)}$ is given as follows.
\begin{itemize}
\item For $\lambda \le \log 2$,     
\begin{equation}  \label{spec density limit lam small}
 \widehat{\rho}^{\rm (dH)}(x)  = \frac{1}{\pi \lambda x} \arg \bigg( \frac{ 2 e^{-\lambda} - \sqrt{\lambda}x  +i \sqrt{  4(1-e^{-\lambda}) e^{-\lambda}-\lambda x^2 } }{  2 e^{-\lambda}+ \sqrt{\lambda} x +i \sqrt{   4  (1-e^{-\lambda}) e^{-\lambda}-\lambda x^2 }  } \bigg)  \mathbbm{1}_{(-a(\lambda),a(\lambda))}(x)
\end{equation}
\item For $\lambda \ge \log 2$,     
\begin{align}
\begin{split}
 \label{spec density limit lam big}
 \widehat{\rho}^{\rm (dH)}(x) &= \frac{1}{\pi \lambda x} \arg \bigg( \frac{ 2 e^{-\lambda} - \sqrt{\lambda}x  +i \sqrt{  4(1-e^{-\lambda}) e^{-\lambda}-\lambda x^2 } }{  2 e^{-\lambda}+ \sqrt{\lambda} x +i \sqrt{   4  (1-e^{-\lambda}) e^{-\lambda}-\lambda x^2 }  } \bigg)  \mathbbm{1}_{(-a(\lambda),a(\lambda))}(x)
 \\
 &\quad + \frac{1}{\lambda |x|} \Big( \mathbbm{1}_{ (-1/\sqrt{\lambda},1/\sqrt{\lambda} }(y) -\mathbbm{1}_{ (-a(\lambda),a(\lambda)) }(y) \Big).
\end{split}
\end{align} 
\end{itemize}
\end{prop}

\begin{proof}
Let 
\begin{equation}\label{G}
G(y):= \int_{ -1/\sqrt{\lambda} }^{ 1/\sqrt{\lambda} } \frac{ \widehat{\rho}^{\rm (dH)}(x) }{ y-x }\,dx, \qquad y \in \C \setminus \R  
\end{equation}
be the resolvent. 
Note that it satisfies 
\begin{align}
\begin{split}
G(y) 
&= \sum_{p=0}^\infty \frac{1}{y^{2p+1}} \int_{ -1/\sqrt{\lambda} }^{ 1/\sqrt{\lambda} }  x^{2p}\, \widehat{\rho}^{ \rm(dH) }(x)  \,dx. 
\end{split}
\end{align}  
Then it follows that 
\begin{align*}
G(y) &  = \frac{1}{y} + \frac{1}{y} \sum_{p=1}^\infty \frac{1}{y^{2p}}  \frac{1}{\lambda^{p+1}}  \binom{2p}{p} \, \int_0^{ 1-e^{-\lambda} } t^{p} (1-t)^{p-1}\,dt 
\\
&=  \frac{1}{y} + \frac{1}{\lambda y}  \int_0^{ 1-e^{-\lambda} } \frac{1}{1-t}  \bigg(  \sum_{p=1}^\infty  \binom{2p}{p}  \Big( \frac{ t (1-t) }{ \lambda y^2  } \Big)^p  \bigg) \,dt. 
\end{align*}
Now, the generalised binomial theorem tells us
\begin{equation}
\sum_{p=0}^\infty \binom{2p}{p} z^p = \frac{1}{\sqrt{1-4z}}, \qquad |z|<1/4. 
\end{equation}
To use this, we note that the condition 
\begin{equation}
\Big| \frac{t(1-t)}{\lambda y^2} \Big| < \frac14, \qquad t \in (0, 1-e^{-\lambda })
\end{equation}
is satisfied for
\begin{equation}
|y| >   \frac{1}{ \sqrt{\lambda } }  \begin{cases}
2\sqrt{(1-e^{-\lambda}) e^{-\lambda}} & \lambda < \log 2,
\smallskip 
\\
1 & \lambda \ge \log 2. 
\end{cases}
\end{equation}
Consequently,
\begin{align}\label{4.28a}
\begin{split}
G(y) & =  \frac{1}{y} + \frac{1}{\lambda y}  \int_0^{ 1-e^{-\lambda} }   \frac{1}{1-t}  \bigg(  \frac{ \sqrt{\lambda }y }{ \sqrt{ \lambda y^2-4t(1-t) } } -1 \bigg) \,dt
 \\
&= \frac{1}{ \lambda y } \int_0^{ 1-e^{-\lambda} }   \frac{1}{1-t}    \Big( 1- \frac{ 4t(1-t) }{ \lambda y^2 } \Big)^{-\frac12}  \,dt. 
\end{split}
\end{align}

To proceed further, suppose that $|y|<1/\sqrt{\lambda}$ is given. 
Then there is range of $t$ values 
\begin{equation}
t \in \Big( \tfrac12 -y^* , \min\{ 1-e^{-\lambda}, \tfrac12+y^* \} \Big), \qquad y^*= \frac12 \sqrt{1-\lambda y^2} 
\end{equation}
for which 
\begin{equation}\label{4.29a}
 1- \frac{ 4t(1-t) }{ \lambda y^2 }  < 0. 
\end{equation}
By the Sokhotski–Plemelj formula applied to (\ref{G}), we have
\begin{equation}\label{SP}
\widehat{\rho}^{ \rm (dH) }(y) = \frac{1}{\pi} \lim_{\epsilon \to 0^+}{\rm Im} \, G(y-i \epsilon).
\end{equation}
Making use of (\ref{4.28a}) and (\ref{4.29a}) then gives
\begin{equation}\label{4.30}
\widehat{\rho}^{ \rm (dH) }(y) = \frac{1}{\pi} \frac{1}{ \lambda y } \int_{ \tfrac12 -y^* }^{  \min\{ 1-e^{-\lambda}, \tfrac12+y^* \} }   \frac{1}{1-t}    \Big( 1- \frac{ 4t(1-t) }{ \lambda y^2 } \Big)^{-\frac12}  \,dt.
\end{equation}

It is possible to evaluate this integral.
First, it follows from (\ref{4.30}) that if $\lambda < \log 2$, 
\begin{align}
\widehat{\rho}^{ \rm (dH) }(y) & =   \frac{1}{ \pi \lambda y } \int_{ \tfrac12 -y^* }^{   1-e^{-\lambda}  }   \frac{1}{1-t}    \Big(  \frac{ 4t(1-t) }{ \lambda y^2 }-1 \Big)^{-\frac12}  \,dt\,  \mathbbm{1}_{ (-a(\lambda),a(\lambda)) }(y), 
\end{align}
while if $\lambda \ge \log 2$, 
\begin{align}\label{4.32}
\begin{split}
\widehat{\rho}^{ \rm (dH) }(y) & =   \frac{1}{ \pi \lambda y } \int_{ \tfrac12 -y^* }^{   1-e^{-\lambda}  }   \frac{1}{1-t}    \Big(  \frac{ 4t(1-t) }{ \lambda y^2 }-1 \Big)^{-\frac12}  \,dt \, \mathbbm{1}_{ (-a(\lambda),a(\lambda)) }(y)
\\
&\quad + \frac{1}{ \pi \lambda y } \int_{ \tfrac12 -y^* }^{   \tfrac12 + y^*  }   \frac{1}{1-t}    \Big(  \frac{ 4t(1-t) }{ \lambda y^2 }-1 \Big)^{-\frac12}  \,dt \, \Big( \mathbbm{1}_{ (-1/\sqrt{\lambda},1/\sqrt{\lambda} }(y) -\mathbbm{1}_{ (-a(\lambda),a(\lambda)) }(y) \Big).
\end{split}
 \end{align}
Now, using computer algebra we can check that 
$$
\int \frac{1}{1-t} (At(1-t)-1)^{-\frac12} \,dt= \frac{1}{i}  \log \bigg( \frac{ 1+\sqrt{A}(1-t) +\sqrt{ 1-At (1-t) }  }{  1-\sqrt{A}(1-t)- \sqrt{ 1-At (1-t) } } \bigg).  $$
Letting $A=4/(\lambda y^2)$, we write
\begin{align*}
  F(t):=  \frac{1}{i}  \log \bigg( \frac{ 1+ \frac{2}{ \sqrt{\lambda} y } (1-t) +i \sqrt{   \frac{4 }{ \lambda y^2} t (1-t)-1 }  }{  1- \frac{2}{ \sqrt{\lambda} y } (1-t)-i \sqrt{  \frac{4}{\lambda y^2} t (1-t)-1 } } \bigg).
\end{align*}
Substituting in (\ref{4.32}) it follows that 
\begin{align}\label{4.34}
\begin{split}
\widehat{\rho}^{ \rm (dH) }(y) & = \frac{1}{\pi \lambda y} \Big( F(1-e^{-\lambda})-F(\tfrac12-y^*) \Big)  \mathbbm{1}_{ (-a(\lambda),a(\lambda)) }(y)
\\
&\quad + \begin{cases}
0  &\textup{if }\lambda \le \log 2,
\smallskip
\\
\displaystyle \frac{1}{\pi \lambda y} \Big(F(\tfrac12+y^*) - F(\tfrac12-y^*) 
 \Big) \Big( \mathbbm{1}_{ (-1/\sqrt{\lambda},1/\sqrt{\lambda} }(y) -\mathbbm{1}_{ (-a(\lambda),a(\lambda)) }(y) \Big)  &\textup{if }\lambda \ge \log 2. 
\end{cases}
\end{split}
\end{align}

The expression \eqref{4.34} can be simplified.
Note that by definition of $y^*$, 
$$
F(\tfrac12-y^*) =  \frac{1}{i}  \log \bigg( \frac{  \sqrt{\lambda} y+  1+ \sqrt{1-\lambda y^2}     }{   \sqrt{\lambda} y-  1- \sqrt{1-\lambda y^2}    } \bigg)  , \qquad  F(\tfrac12+y^*) = \frac{1}{i}  \log \bigg( \frac{ \sqrt{\lambda} y + 1- \sqrt{1-\lambda y^2}    }{  \sqrt{\lambda} y- 1+ \sqrt{1-\lambda y^2}   } \bigg) .
$$
Therefore it follows that 
\begin{equation}
F(\tfrac12+y^*) - F(\tfrac12-y^*) = \pi. 
\end{equation}
On the other hand, we have 
\begin{align*}
F(1-e^{-\lambda}) 
 &=  \frac{1}{i}  \log \bigg( \frac{ \sqrt{\lambda} y + 2 e^{-\lambda} +i \sqrt{   4  (1-e^{-\lambda}) e^{-\lambda}-\lambda y^2 }  }{  \sqrt{\lambda}y -2 e^{-\lambda}-i \sqrt{  4(1-e^{-\lambda}) e^{-\lambda}-\lambda y^2 } } \bigg)
 \\
 &= \frac{1}{i } \log \bigg( \frac{ 1+y\sqrt{\lambda} }{ 1-y\sqrt{\lambda} } \bigg)^{1/2} + \arg \bigg( \frac{ \sqrt{\lambda} y + 2 e^{-\lambda} +i \sqrt{   4  (1-e^{-\lambda}) e^{-\lambda}-\lambda y^2 }  }{  \sqrt{\lambda}y -2 e^{-\lambda}-i \sqrt{  4(1-e^{-\lambda}) e^{-\lambda}-\lambda y^2 } } \bigg). 
\end{align*}
Note here that 
\begin{align*}
 \log \bigg( \frac{ 1+y\sqrt{\lambda} }{ 1-y\sqrt{\lambda} } \bigg)^{1/2}  -  \log \bigg( \frac{  \sqrt{\lambda} y+  1+ \sqrt{1-\lambda y^2}     }{   \sqrt{\lambda} y-  1- \sqrt{1-\lambda y^2}    } \bigg)  = \pi i. 
\end{align*}
Therefore it follows that 
\begin{equation}
F(1-e^{-\lambda}) - F(\tfrac12-y^*) = \arg \bigg( \frac{ 2 e^{-\lambda} - \sqrt{\lambda}y  +i \sqrt{  4(1-e^{-\lambda}) e^{-\lambda}-\lambda y^2 } }{  2 e^{-\lambda}+ \sqrt{\lambda} y +i \sqrt{   4  (1-e^{-\lambda}) e^{-\lambda}-\lambda y^2 }  } \bigg).
\end{equation}
Combining all of the above in \eqref{4.34}, the proof is complete. 
\end{proof}

\begin{proof}[Proof of Theorem~\ref{Thm_genus one density}]

By Theorem~\ref{Thm_genus one}, we have
\begin{equation}
\frac{1}{N} \int_{ 1/\sqrt{q} }^{ 1/\sqrt{q} } \widetilde{\rho}_N^{ \rm (dH) }(x)\,d_q x = \mathcal{M}_{2p,0}^{ \rm (dH) } + \frac{ \mathcal{M}_{2p,1}^{ \rm (dH) } }{ N^2  } +O (N^{-4}  ).
\end{equation}
The leading order follows from Proposition~\ref{Prop_density qGUE} since 
\begin{equation}
\widetilde{\rho}_{(0)}(x)  = \frac{1}{\sqrt{\lambda}} \,  \widehat{\rho}^{ \rm (dH) }\Big(\frac{x}{\sqrt{\lambda} } \Big) . 
\end{equation}

Write $\eta= e^{-\lambda}$. 
By Theorem~\ref{Thm_genus one},  
\begin{align}
\mathcal{M}_{2p,1}^{ \rm (dH) } 
&= - \log \eta  \bigg(  \frac{p}{6} I_{1-\eta }(p+1,p) + \frac{(2p-1)!}{6(p-1)!^2} \eta^p (1-\eta)^{p-1} \Big( 2+p-(2p+1)\eta \Big)  \bigg)   . 
\end{align}
We need to compute 
\begin{align}
\widetilde{G}_{(1)}(x):= \frac{1}{x}+\frac{1}{x} \sum_{p=1}^\infty \frac{ \mathcal{M}_{2p,1}^{ \rm (dH) } }{ x^{2p} }. 
\end{align}

Note that 
\begin{equation}
\frac{1}{1-\eta} \, \sum_{p=1}^\infty  \frac{(2p-1)!}{6(p-1)!^2}  \Big( \frac{\eta (1-\eta)}{x^2} \Big)^{p} \Big( 2+p-(2p+1)\eta \Big)=  \frac{\eta(1-\eta) (x^2-2\eta) }{ 2x^4 } \Big( 1- \frac{4\eta (1-\eta)}{x^2} \Big)^{-5/2} . 
\end{equation}
We also have 
\begin{align}
\begin{split}
 \sum_{p=1}^\infty \frac{p}{6} \frac{1}{x^{2p}} I_{1-\eta}(p+1,p) &= \frac16   \int_0^{1-\eta} \frac{1}{1-t} \bigg( \sum_{p=1}^\infty \frac{(2p)!}{(p-1)!^2} \Big( \frac{t(1-t)}{x^2} \Big)^p \bigg) \,dt  
 \\
 &=  \frac{1}{3x^4}   \int_0^{1-\eta}   t \Big( x^2+2t(1-t) \Big) \Big( 1- \frac{4t (1-t)}{x^2} \Big)^{-5/2}    \,dt .  
\end{split}
\end{align}
Then 
\begin{align}
\begin{split}
\widetilde{G}_{(1)}(x)& = \frac{1}{x} - \lambda \frac{ e^{-\lambda} (1-e^{-\lambda})}{2 }   (x^2-2e^{-\lambda} )  \Big(   x^2-4e^{-\lambda} (1-e^{-\lambda}) \Big)^{-5/2}
\\
&\quad - \frac{ \lambda }{3 }  \int_0^{1-e^{-\lambda}}   t \Big( x^2+2t(1-t) \Big) \Big( x^2-4t (1-t) \Big)^{-5/2}    \,dt .  
\end{split}
\end{align}
Note that $x^2-4t(1-t)<0$ if  
\begin{equation}
t \in  \Big(\tfrac12-x^* , \min\{ 1-e^{-\lambda}, \tfrac12+ x^*  \} \Big), \qquad x^*= \frac12 \sqrt{1-x^2}. 
\end{equation}
Then the Sokhotski–Plemelj (\ref{SP}) formula completes the proof. 
\end{proof}


\begin{thebibliography}{999}

\bibitem{AAT19} S. Alwhishi, R. S. Adigüzel and M. Turan, \emph{On the limit of discrete $q$-Hermite I polynomials}, Commun. Fac. Sci. Univ. Ank. Ser. A1 Math. Stat. \textbf{68} (2019), 2272--2282.

\bibitem{ACC23}  Y. Ameur, C. Charlier and J. Cronvall, \emph{Free energy and fluctuations in the random normal matrix model with spectral gaps}, arXiv:2312.13904. 

\bibitem{AGZ09} G. W. Anderson, A.~Guionnet, and O.~Zeitouni, \emph{An introduction to random matrices}, Cambridge University Press, Cambridge, 2009.

\bibitem{ABGS21} T. Assiotis, B. Bedert, M.A. Gunes, and A. Soor, \emph{On a distinguished family of random variables and Painlev\'e equations}, Probab. Math. Phys. \textbf{2} (2021), 613--642.

\bibitem{ABGS21a} T. Assiotis, B. Bedert, M. Gunes and A. Soor, \emph{Moments of generalized Cauchy random matrices and continuous-Hahn polynomials}, Nonlinearity, \textbf{34} (2021), 4923--4943.

\bibitem{AKW22} T.~Assiotis, J. P. Keating and J. Warren, \emph{On the joint moments of the characteristic polynomials of random unitary matrices},  Int. Math. Res. Not. \textbf{2022} (2022), 14564--14603. 


\bibitem{Be97} C. W. J.~Beenakker  \emph{Random-matrix theory of quantum transport}, Rev. Mod. Phys. \textbf{69} (1997), 731--808.


\bibitem{BG16} A.~Borodin and V.~Gorin, \emph{Lectures on integrable probability}. In: V.~Sidoravicius and S.~Smirnov, (eds) Probability and Statistical Physics in St. Petersburg, Proc. Sympos. Pure Math., vol. 91, pp. 155--214. Amer. Math. Soc., Providence, 2016.

\bibitem{BS09} A. Borodin and E. Strahov, \emph{Correlation kernels for discrete symplectic and orthogonal ensembles}, Comm. Math. Phys. \textbf{286} (2009), 933--977.


\bibitem{BGG17} A. Borodin, V. Gorin and A. Guionnet, \emph{Gaussian asymptotics of discrete $\beta$-ensembles}, Publ. math. IHES \textbf{125} (2017), 1--78.


\bibitem{BIPZ78} E. Br\'{e}zin, C. Itzykson, G. Parisi and J. B. Zuber, \emph{Planar diagrams}, Comm. Math. Phys. \textbf{59} (1978), 35--51.


\bibitem{BK18} A. Bufetov and A. Knizel, \emph{Asymptotics of random domino tilings of rectangular Aztec diamonds}, Ann. Inst. Henri Poincaré Probab. Stat. \textbf{54} (2018), 1250--1290.


\bibitem{By23} S.-S. Byun, \emph{Harer-Zagier type recursion formula for the elliptic GinOE}, arXiv:2309.11185.

\bibitem{BF23} S.-S. Byun and P. J. Forrester, \emph{Spectral moments of the real Ginibre ensemble}, Ramanujan J. (to appear), arXiv:2312.08896.

\bibitem{BKS23} S.-S. Byun, N.-G. Kang and S.-M. Seo, \emph{Partition functions of determinantal and Pfaffian Coulomb gases with radially symmetric potentials}, Comm. Math. Phys. \textbf{401} (2023), 1627--1663.

\bibitem{Co21} P. Cohen, \emph{Moments of discrete classical $q$-orthogonal polynomial ensembles}, arXiv:2112.02064. 


\bibitem{CCO20} P. Cohen, F. D. Cunden and N. O'Connell, \emph{Moments of discrete orthogonal polynomial ensembles}, Electron. J. Probab. \textbf{25} (2020), Paper No. 72, 19 pp.



\bibitem{CJKK23} S. Corteel, B. Jonnadula, J. P. Keating and J. S. Kim, \emph{Lecture hall graphs and the Askey scheme}, arXiv:2311.12761. 


\bibitem{CKS16} S. Corteel, J. S. Kim and D. Stanton, \emph{Moments of orthogonal polynomials and combinatorics}, Recent trends in combinatorics, 545–578, IMA Vol. Math. Appl., 159, Springer, 2016.





\bibitem{CRS06} J. B. Conrey, M. O. Rubinstein and N. C. Snaith, \emph{Moments of the derivative of characteristic polynomials with an application to the {R}iemann zeta function}, Comm. Math. Phys. \textbf{267} (2006), 611--629.



\bibitem{CDO21} F. D. Cunden, A. Dahlqvist and N. O’Connell, \emph{Integer moments of complex Wishart matrices and Hurwitz numbers}, Ann. Inst. Henri Poincaré Comb. Phys. Interact. \textbf{8} (2021), 243--268.

\bibitem{CMOS19} F. D. Cunden, F. Mezzadri, N. O'Connell and N. Simm, \emph{Moments of random matrices and hypergeometric orthogonal polynomials}, Comm. Math. Phys. \textbf{369} (2019), 1091--1145.


\bibitem{DDMS19} D. S.~Dean, P.~Le Doussal, S. N.~Majumdar and G.~Schehr, \emph{Non-interacting fermions in a trap and random matrix theory}, J. Phys. A \textbf{52} (2019), 144006.





\bibitem{Fl80} P. Flajolet, \emph{Combinatorial aspects of continued fractions}, Discrete Math. \textbf{32} (1980), 125--161.


\bibitem{Fo10} P. J. Forrester, \emph{Log-gases and random matrices}, Princeton University Press, Princeton, NJ, 2010.

\bibitem{Fo21} P. J. Forrester, \emph{Global and local scaling limits for the $\beta=2$ Stieltjes-Wigert random matrix ensemble}, Random Matrices Theory Appl. \textbf{11} (2022), 2250020. 

\bibitem{Fo21a} P. J. Forrester, \emph{Differential identities for the structure function of some random matrix ensembles}, J. Stat. Phys. \textbf{183} (2021), 33. 


\bibitem{Fo22} P. J.~Forrester, \emph{Joint moments of a characteristic polynomial and its derivative for the circular $\beta$ ensemble}, Probab. Math. Phys. \textbf{3} (2022), 145--170. 


\bibitem{FH22} P. J. Forrester and S.-H Li, \emph{Rate of convergence at the hard edge for various Pólya ensembles of positive definite matrices}, Integral Transforms Spec. Funct. \textbf{33} (2022), 466--484. 


\bibitem{FLSY23} P. J. Forrester, S.-H. Li, B.-J. Shen and G.-F. Yu, \emph{$q$-Pearson pair and moments in $q$-deformed ensembles}, Ramanujan J. \textbf{60} (2023), 195--235.

\bibitem{FR21} P. J. Forrester and A. Rahman, \emph{Relations between moments for the Jacobi and Cauchy random matrix ensembles}, J. Math. Phys. \textbf{62} (2021), 073302.



\bibitem{FLD16} Y. V.~Fyodorov and P.~Le~Doussal, \emph{Moments of the position of the maximum for GUE characteristic polynomials and for log-correlated Gaussian processes}, J. Stat. Phys. \textbf{164} (2016), 190--240.


\bibitem{GGR21} M. Gisonni, T. Grava and G. Ruzza, \emph{Jacobi ensemble, Hurwitz numbers and Wilson polynomials}, Lett. Math. Phys. \textbf{111} (2021), no. 3, Paper No. 67, 38 pp. 
 

\bibitem{HSS92} P. J. Hanlon, R. P. Stanley and J. R. Stembridge, \emph{Some combinatorial aspects of the spectra of normally distributed random matrices}, Contemp. Math. \textbf{138} (1992), 151--174.

\bibitem{HZ86} J. Harer and D. Zagier, \emph{The Euler characteristic of the moduli space of curves}, Invent. Math. \textbf{85} (1986), 457--485.


\bibitem{HT03} U. Haagerup and S. Thorbjørnsen, \emph{Random matrices with complex Gaussian entries}, Expo. Math. \textbf{21} (2003), 293--337.


\bibitem{KSS09} B. A. ~Khoruzhenko, D. V. ~Savin, and  H.-J.  Sommers, \emph{Systematic approach to statistics of conductance and shot-noise in chaotic cavities}, Phys. Rev. B \textbf{80} (2009), 125301.

\bibitem{ISV87} M. E. H. Ismail, D. Stanton and G. Viennot, \emph{The combinatorics of $q$-Hermite polynomials and the Askey-Wilson integral}, European J. Combin. \textbf{8} (1987), 379--392.

\bibitem{Jo01}  K. Johansson, \emph{Discrete orthogonal polynomial ensembles and the Plancherel measure}, Ann. Math. \textbf{153} (2001), 259--296.


\bibitem{JKM20} B.~Jonnadula, J. P.~Keating and F.~Mezzadri, \emph{Symmetry function theory and unitary invariant ensembles}, J. Math. Phys. \textbf{62} (2021), 093512.


\bibitem{KS03} J. P.~Keating and N. C.~Snaith, \emph{Random matrices and L-functions}, J. Phys. A \textbf{36} (2003), 2859--2881.


\bibitem{Ke99} S. V. Kerov, \emph{Rooks on Ferrers boards and matrix integrals}, J Math Sci \textbf{96} (1999), 3531--3536.


\bibitem{KLS10} R. Koekoek, P. A. Lesky and R. F. Swarttouw, \emph{Hypergeometric Orthogonal Polynomials and Their $q$-Analogues}, Springer Monographs in Mathematics, Springer-Verlag Berlin Heidelberg, 2010.



\bibitem{Le04} M. Ledoux. \emph{Differential operators and spectral distributions of invariant ensembles from the classical orthogonal polynomials. The continuous case}, Electron. J. Probab. \textbf{9} (2004) 177–208.

\bibitem{Le05} M. Ledoux, \emph{Differential operators and spectral distributions of invariant ensembles from the classical orthogonal polynomials: the discrete case}, Electron. J. Probab. \textbf{10} (2005), 1116--1146. 

\bibitem{Le09} M. Ledoux, \emph{A recursion formula for the moments of the Gaussian orthogonal ensemble}, Ann. Inst. Henri Poincaré Probab. Stat. \textbf{45} (2009), 754--769.
 

\bibitem{LSYF23} S.-H. Li, B.-J. Shen, G.-F. Yu and  P. J. Forrester, \emph{Discrete orthogonal ensemble on the exponential lattices}, arXiv:2206.08633.


\bibitem{Li12} K. Liechty, \emph{Non-intersecting Brownian motions on the half-line and discrete Gaussian orthogonal polynomials}, J. Stat. Phys. \textbf{147} (2012), 582--622.


\bibitem{Ma14} M. Marino, \emph{Lectures on non-perturbative effects in large N gauge theories, matrix models and strings}, Fortsch. Phys. \textbf{62} (2014) 455--540.


\bibitem{MS11} F. Mezzadri and N. Simm, \emph{Moments of the transmission eigenvalues, proper delay times, and random matrix theory I}, J. Math. Phys. \textbf{52} (2011), 103511.

\bibitem{MS12} F. Mezzadri and N. Simm, \emph{Moments of the transmission eigenvalues, proper delay times and random matrix theory II}, J. Math. Phys. \textbf{53} (2012), 053504.

\bibitem{MS13} F. Mezzadri and N. Simm, \emph{$\tau$-function theory of quantum chaotic transport with $\beta = 1, 2, 4$}, Comm. Math. Phys. \textbf{324} (2013), 465.


\bibitem{MRW15} F.~Mezzadri, A.K. Reynolds and B.~Winn, \emph{Moments of the eigenvalue densities and of the secular coefficients of $\beta$-ensembles}, Nonlinearity \textbf{30} (2017), 1034--1057.
 
 

\bibitem{MPS20} A. Morozov, A. Popolitov and Shakirov, \emph{Quantization of Harer-Zagier formulas}, Phys. Lett. B \textbf{811} (2020), 135932. 


\bibitem{NS13} P. Njionou Sadjang, \emph{Moments of classical orthogonal polynomials}, Ph.D. Thesis, Universit\"at Kassel, 2013.


\bibitem{No22} M. Novaes, \emph{Time delay statistics for finite number of channels in all symmetry classes}, Europhys. Lett. \textbf{139} (2022), 21001.

\bibitem{Ol21} G. Olshanski, \emph{Macdonald-level extension of beta ensembles and large-$N$ limit transition}, Comm. Math. Phys. \textbf{385} (2021), 595--631. 


\bibitem{NIST}  F. W. J. Olver, D. W. Lozier, R. F. Boisvert, and C. W. Clark., eds. \emph{NIST Handbook of Mathematical Functions}, Cambridge: Cambridge University Press, 2010.


\bibitem{On81} E. Onofri,  \emph{SU(N) Lattice gauge theory with Villain's action}, Nuovo Cim. A  \textbf{66} (1981), 293--318.


\bibitem{RF21} A. Rahman and P. J. Forrester, \emph{Linear differential equations for the resolvents of the classical matrix ensembles},
Random Matrices Theory Appl. \textbf{10} (2021), 2250003.


\bibitem{STGK16} S. Sedighi, A. Taherpour, S. Gazor and T. Khattab, \emph{Eigenvalue-based multiple antenna spectrum sensing: Higher order
moments}, IEEE Trans. Wireless Commun. \textbf{16} (2016), 1168--1184.



\bibitem{Ti39} E. C. Titchmarsh, \emph{Theory of functions}, Oxford University Press, London, 1939.

\bibitem{TV04} A. M. Tulino and S. Verdu \emph{Random matrix theory and wireless communications}, Foundations and Trends in Communcations and Information Theory, vol. 1, Now Publisher, 2004, pp. 1--182.

\bibitem{Ve14} P. S. Venkataramana, \emph{On $q$-analogs of some integrals over GUE}, arXiv:1407.79602.


\bibitem{Vi00} X. G. Viennot, \emph{A combinatorial interpretation of the quotient-difference algorithm}, Formal Power Series and Algebraic Combinatorics (Moscow, 2000), pages 379--390. Springer, Berlin, 2000.

\bibitem{Wi12} B. Winn, \emph{Derivative moments for characteristic polynomials from CUE}, Comm. Math. Phys. \textbf{315} (2012), 532--562.

\bibitem{Wi55} E. Wigner, \emph{Characteristic vectors of bordered matrices with infinite dimensions}, Ann. of Math. \textbf{62} (1955), 548--564.

\bibitem{WF14} N. Witte and P. J. Forrester, \emph{Moments   of the Gaussian $\beta$ ensembles and the large $N$ expansion of the densities}, J. Math. Phys. \textbf{55} (2014), 083302. 
 
 \end{thebibliography}
\end{document}